\documentclass[11pt]{amsart}

\usepackage{amssymb}
\usepackage[all,cmtip]{xy}
\usepackage{mathrsfs,extpfeil}

\newtheorem{theorem}{Theorem}[section]
\newtheorem{lem}[theorem]{Lemma}
\newtheorem{cor}[theorem]{Corollary}
\newtheorem{prop}[theorem]{Proposition}
\newtheorem{conj}[theorem]{Conjecture}

\theoremstyle{definition}
\newtheorem{defi}[theorem]{Definition}
\newtheorem{example}[theorem]{Example}

\theoremstyle{remark}
\newtheorem{rem}[theorem]{Remark}

\newcommand{\BC}{\mathbb{C}}            
\newcommand{\BZ}{\mathbb{Z}}             
\newcommand{\BN}{\mathbb{N}}            
\newcommand{\dstirling}[2]{\genfrac{[}{]}{0pt}{0}{#1}{#2}}   

\newcommand{\Gaff}{\hat{\mathfrak{g}}}        
\newcommand{\Glie}{\mathfrak{g}}             
\newcommand{\Osh}{\mathcal{O}^{\mathrm{sh}}}
\newcommand{\Hlie}{\mathfrak{h}}          

\newcommand{\qaf}{U_q(\hat{\mathfrak{g}})}    
\newcommand{\Borel}{U_q(\hat{\mathfrak{b}})}   
\newcommand{\CU}{\mathcal{U}}            
\newcommand{\BT}{\mathbf{T}}    
\newcommand{\BA}{\mathbf{A}}    
\newcommand{\SA}{\mathscr{A}}
\newcommand{\CR}{\mathcal{R}}    
\newcommand{\barR}{\overline{\mathcal{R}}}    

\newcommand{\BQ}{\mathbf{Q}}                
\newcommand{\BP}{\mathbf{P}}            
\newcommand{\wt}{\mathrm{wt}}         

\newcommand{\Be}{\mathbf{e}}         
\newcommand{\Bf}{\mathbf{f}}  
\newcommand{\Bn}{\mathbf{n}}
\newcommand{\Bm}{\mathbf{m}}
\newcommand{\Bp}{\mathbf{a}}
\usepackage{color}

\allowdisplaybreaks                
\title{Jordan-H\"older property for shifted quantum affine algebras}
\author{David Hernandez}
\address{DH: Université Paris Cité and Sorbonne Université, CNRS, IMJ-PRG, F-75006, Paris, France}
\email{david.hernandez@imj-prg.fr}
\author{Huafeng Zhang}
\address{HZ: CNRS, UMR 8524-Laboratoire Paul Painlev\'e, Univ. Lille, F-59000 Lille, France}
\email{huafeng.zhang@univ-lille.fr}
\begin{document}
\begin{abstract} We prove that finite length representations of shifted quantum affine algebras in category $\mathcal{O}^{\mathrm{sh}}$ are stable by fusion product. 
This implies that in the topological Grothendieck ring $K_0(\mathcal{O}^{\mathrm{sh}})$ the Grothendieck group of finite length representations forms a non-topological subring. We also conjecture this subring is isomorphic to the cluster algebra discovered in \cite{GHL}.
In the course of our proofs, we establish that any simple representation in category $\mathcal{O}^{\mathrm{sh}}$ descends to a truncation for certain truncation parameters as conjectured in \cite{H} in terms of Langlands dual $q$-characters. \end{abstract}

\maketitle
\tableofcontents

\section{Introduction}

Let $\Glie$ be a simple complex finite dimensional Lie algebra, and let $U_q(\Gaff)$ be the associated non-twisted quantum affine algebra for $q$ a nonzero complex number which is not a root of unity. 

In their study of quantized $K$-theoretic Coulomb branches of $3d$ $\mathcal{N} = 4$ SUSY quiver gauge theories, Finkelberg and Tsymbaliuk \cite{FT} introduced a family of associative algebras $\CU_{\mu}(\Gaff)$ called {\it shifted quantum affine algebras}. They depend on a shift parameter $\mu$ that is an integral coweight of $\mathfrak{g}$ and also implicitly on the deformation parameter $q$. When $\mu = 0$, the algebra $\CU_0(\Gaff)$ is a central extension of $U_q(\Gaff)$ and has essentially the same representation theory. When $\mu \not = 0$, the representation theory becomes very different. In \cite{H}, the first author introduced a category $\mathcal{O}_\mu$ containing infinite-dimensional representations, and showed that the topological Grothendieck group of the direct sum $\Osh$ of categories $\mathcal{O}_\mu$ for all integral coweights has a natural ring structure coming from an operation on simple representations called fusion product. 

In the present paper, we establish the following Jordan--H\"older property for shifted quantum affine algebras: all fusion products of simple representations in the category $\Osh$ are of finite length. 

Our proof is partly based on the study of remarkable quotients of shifted quantum affine algebras, the truncated shifted quantum affine algebras (or truncations), also introduced in \cite{FT} to realize quantized $K$-theoretical Coulomb branches. In \cite{H} it was proved that each truncation has a finite number of simple representations up to isomorphism in the category $\Osh$; this is the Jordan--H\"older property for truncations. One crucial point in our present proof is to establish that any simple representation in category $\Osh$ factorizes through a truncation. This result, which has some interest on its own, is obtained in this paper by using certain polynomiality results related to the universal R-matrix of the quantum affine algebra $\qaf$ as a quasi-triangular Hopf algebra. In more detail, truncations are defined in terms of certain power series with coefficients in shifted quantum affine algebras. We interpret these series as limits of transfer matrices obtained from the universal R-matrix.

Our main results, the Jordan--H\"older property and the truncation property, were established previously \cite{HZ} in a different context for the so-called shifted Yangians. These algebras appear much earlier in the context of quantized Coulomb branches \cite{BFN} and independently in the study of affine Grassmannian slices \cite{KWWY}.  The proof in \cite{HZ} relied in an important way on an extension of the Drinfeld--Jimbo coproduct of the ordinary Yangian, the shifted coproduct constructed in \cite{coproduct}. Such a shifted coproduct is conjectural for shifted quantum affine algebras beyond type A. For that reason, we had to write a different proof. Our new approach has the advantage that the truncation parameters can be computed explicitly.

We also apply our methods to a natural subcategory of $\Osh$, the category $\widehat{\mathcal{O}}$ of representations of the quantum affine algebra. This is a tensor category whose monoidal structure is defined by the ordinary Drinfeld--Jimbo coproduct instead of the fusion product. We refine the truncation property for category $\widehat{\mathcal{O}}$: any tensor product of simple representations in this category factorizes through a truncation. This result seems to be unknown for Yangians. A key step in its proof is the polynomiality property for the Drinfeld--Jimbo coproduct observed in \cite{Z}.  

We obtain several applications of our result.

First, we show that the topological Grothendieck ring $K_0(\Osh)$ admits an ordinary subring $K_0(\mathcal{O}^{\rm sh,f})$ formed of isomorphism classes of finite length representations. We think this is an interesting ring that deserves further study.

Secondly,  in \cite{GHL} Geiss--Hernandez--Leclerc constructed a completed cluster algebra structure on the topological Grothendieck ring $K_0(\Osh_{\BZ})$ of a certain subcategory $\Osh_{\BZ}$. They conjectured that the cluster variables correspond to simple classes. Using our result, we can partly reformulate this conjecture: the cluster algebra of \cite{GHL} is isomorphic to the subring $K_0(\mathcal{O}_{\BZ}^{\rm sh,f})$ of isomorphism classes of finite length representations.

Eventually, a conjectural classification of simple modules of truncations was formulated in \cite{H} in terms of Langlands dual $q$-characters (obtained as limits of elements in the deformed $\mathcal{W}$-algebra). We prove that any simple module in $\Osh$ actually descends to a truncation for certain truncation parameters as predicted by the conjecture.

This paper is organized as follows.

In Section \ref{sec: Yangian}, we give reminders on shifted quantum affine algebras and their representations.

In Section \ref{sec: FT}, we give reminders on truncated shifted quantum affine algebras, in their different versions, the original adjoint truncations of \cite{FT} and some slight modifications that we introduce: the simply connected truncations and the intermediate truncations, underlying their different advantages. We also recall from \cite{H} Jordan--H\"older property for adjoint truncations (Theorem \ref{thm: JH}).

In Section \ref{TruncRmat}, we recall a result of \cite{Z} on the polynomiality of $R$-matrices  and then we establish a polynomiality of $A$-series which are used to define truncations.

In Section \ref{invdom}, we prove one of our main results: any irreducible module in the category $\mathcal{O}_{\mu}$ factorizes through a simply connected truncation whose truncation parameters are made explicit (Theorem  \ref{thm: truncation shifted}). A similar statement is established for tensor products of several irreducible modules in category $\widehat{\mathcal{O}}$. As an intermediate result, we obtain in Proposition \ref{prop: key} a sufficient condition for a module in category $\Osh$ to be of finite length. 

In Section \ref{appli}, we obtain several applications from our main results: the existence of subrings of finite length representations (Corollaries \ref{subring}, \ref{fls}, \ref{subring cluster}),  their relation to the cluster algebras of \cite{GHL} (Conjecture \ref{isomclus}) and then the proof of the descent to truncation in accordance with \cite[Conjecture 12.2]{H}(Corollary \ref{clastrunc}).

\medskip

\noindent {\bf Acknowledgments:} We would like to thank C. Geiss and B. Leclerc for useful discussions, and the referee for a careful reading of this paper and thorough comments. The second author acknowledges the support of the CDP C2EMPI, as well as the French State under the France-2030 programme, the University of Lille, the Initiative of Excellence of the University of Lille, the European Metropolis of Lille for their funding and support of the R-CDP-24-004-C2EMPI project.

\section{Generalities on shifted quantum affine algebras} \label{sec: Yangian}
In this section we review the basic properties of quantum affine algebras, their shifted versions, and their representation theory. The ground field is $\BC$, and $\BN := \BZ_{\geq 0}$.

Fix $\Glie$ to be a finite-dimensional simple Lie algebra. Let $\Hlie$ be a Cartan subalgebra of $\Glie$, and $I := \{1,2,\cdots,r\}$ be the set of Dynkin nodes. The dual Cartan subalgebra $\Hlie^*$ admits a basis of {\it simple roots} $(\alpha_i)_{i \in I}$ and a non-degenerate invariant symmetric bilinear form $(,): \Hlie^* \times \Hlie^* \rightarrow \BC$ normalized in such a way that the $d_i := \frac{(\alpha_i,\alpha_i)}{2}$ for $i \in I$ are coprime positive integers in $\{1,2,3\}$. We have the Cartan matrix $(c_{ij})_{i,j\in I}$ and the symmetrized Cartan matrix $(b_{ij})_{i,j\in I}$ whose entries are integers defined by
$$c_{ij} := \frac{2(\alpha_i,\alpha_j)}{(\alpha_i,\alpha_i)},\quad b_{ij} := (\alpha_i,\alpha_j) = d_i c_{ij}. $$ 
Let $r^{\vee}  := \max_{i\in I}(d_i)$ denote the lacing number, $h^{\vee}$ the dual Coxeter number, and $i \mapsto \overline{i}$ the involution on $I$ induced by the longest Weyl group element, all with respect to the Lie algebra $\Glie$.
Let $(\varpi_i^{\vee})_{i \in I}$ be the basis of $\Hlie$ dual to the basis $(\alpha_i)_{i\in I}$ of $\Hlie^*$ with respect to the natural pairing $\langle, \rangle: \Hlie \times \Hlie^* \rightarrow \BC$; the $\varpi_i^{\vee}$ are called {\it fundamental coweights}.
We shall need the coweight lattice in $\Hlie$, the root lattice in $\Hlie^*$ and its cones 
$$\BP^{\vee} := \bigoplus_{i\in I} \BZ \varpi_i^{\vee}, \quad  \BQ := \bigoplus_{i\in I} \BZ \alpha_i,\quad \BQ_+ := \sum_{i\in I} \BN \alpha_i,\quad \BQ_- := -\BQ_+.  $$
 By a coweight we mean an element $\mu$ of the coweight lattice $\BP^{\vee}$, that is, a $\BZ$-linear combination $\sum_{i\in I} m_i \varpi_i^{\vee}$ of the fundamental coweights. Notice that $m_i = \langle \mu, \alpha_i \rangle$. We say that $\mu$ is {\it dominant} if all the $m_i$ are non-negative. We say that $\mu$ is {\it antidominant} if $-\mu$ is dominant.

Fix $q \in \BC^{\times}$, which is not a root of unity. For $t \in \BZ$ and $n \in \BN$ set
$$  [t]_q := \frac{q^t-q^{-t}}{q-q^{-1}}, \quad \dstirling{t}{n}_q := \prod_{m=1}^n [t-m+1]_q.  $$
Set $q_i := q^{d_i}$ for $i \in I$.
 Define three $I\times I$-matrices $B(q), C(q)$ and $D(q)$ by the formulas 
  $$ B_{ij}(q) := [b_{ij}]_q,\quad C_{ij}(q) := [c_{ij}]_{q_i}, \quad D_{ij}(q) := \delta_{ij} [d_i]_q. $$
 These are invertible matrices and moreover $B(q) = D(q) C(q)$.
 
\subsection{The quantum affine algebra}
Let $\theta = \sum_{i=1}^r a_i \alpha_i \in \BQ_+$ be the highest positive root of $\mathfrak{g}$.
We enlarge the Cartan matrix $(c_{ij})_{1\leq i,j \leq r}$ to an affine-type Cartan matrix $(c_{ij})_{0\leq i,j\leq r}$ as follows:
$$c_{00} := 2, \quad c_{i0} = -\frac{2(\alpha_i, \theta)}{(\alpha_i, \alpha_i)},\quad  c_{0i} := -\frac{2(\theta, \alpha_i)}{(\theta, \theta)}\quad \mathrm{for}\ 1 \leq i \leq r. $$
 Define $q_0 := q^{d_0}$ where $d_0 := \frac{(\theta,\theta)}{2}$.
 
The {\it quantum affine algebra} $\qaf$ in the Drinfeld--Jimbo realization \footnote{To be precise, it is the Drinfeld--Jimbo quantum group attached to the affine Cartan matrix $(c_{ij})_{0\leq i,j\leq r}$ of zero central charge and without derivation operator.} is the associative algebra generated by $e_i, f_i, k_i^{\pm 1}$ for $0 \leq i \leq r$ and with relations for $0 \leq i, j \leq r$:
\begin{gather*}
 k_ik_i^{-1} = 1 = k_i^{-1}k_i,\quad k_i k_j = k_j k_i, \quad k_0 k_1^{a_1} k_2^{a_2} \cdots k_r^{a_r} = 1, \\
k_i e_j = q_i^{c_{ij}} e_j k_i,\quad k_i f_j = q_i^{-c_{ij}} f_j k_i, \quad [e_i, f_j] = \delta_{ij} \frac{k_i - k_i^{-1}}{q_i-q_i^{-1}}, \\
\sum_{s=0}^{1-c_{ij}} (-1)^s \dstirling{1-c_{ij}}{s}_{q_i} x_i^{1-c_{ij}-s} x_j x_i^s = 0  \quad \textrm{if $i \neq j$ and $x \in \{e, f\}$}.
\end{gather*}
The algebra $\qaf$ has a Hopf algebra structure with the coproduct given by:
\begin{equation*} 
\Delta(e_i) = e_i \otimes 1 + k_i \otimes e_i,\quad \Delta(f_i) = 1 \otimes f_i + f_i \otimes k_i^{-1},\quad \Delta(k_i) = k_i \otimes k_i.
\end{equation*}
The {\it Borel subalgebra}, denoted by $\Borel$, is the subalgebra generated by $e_i, k_i^{\pm 1}$ for $0 \leq i \leq r$. It is a Hopf subalgebra of $\qaf$.

The quantum affine algebra $\qaf$ has a second presentation, called Drinfeld new realization \cite{Dr, Beck, Damiani, Damiani2}. Its generators are $x_{i,m}^{\pm}, \phi_{i,m}^{\pm}$ for $(i,m) \in I \times \BZ$. Its defining relations are as follows for $(i,j,m,n) \in I^2 \times \BZ^2$ and $\varepsilon \in \{+, -\}$:
\begin{gather} 
\phi_{i,m}^+ = 0 \ \mathrm{if}\ m < 0,\quad \phi_{i,m}^- = 0 \ \mathrm{if}\ m > 0, \quad \phi_{i,0}^+ \phi_{i,0}^- = 1,  \label{Drinfeld rel: initial}  \\
[\phi_{i,m}^{\pm}, \phi_{j,n}^{\varepsilon}] = 0, \quad [x_{i,m}^+, x_{j,n}^-] = \delta_{ij} \frac{\phi_{i,m+n}^+-\phi_{i,m+n}^-}{q_i-q_i^{-1}}, \label{Drinfel rel: Cartan}  \\
\phi_{i,m+1}^{\varepsilon} x_{j,n}^{\pm} -q^{\pm b_{ij}} \phi_{i,m}^{\varepsilon} x_{j,n+1}^{\pm} = q^{\pm b_{ij}} x_{j,n}^{\pm} \phi_{i,m+1}^{\varepsilon} - x_{j,n+1}^{\pm} \phi_{i,m}^{\varepsilon},  \label{Drinfeld rel: Drinfeld-Cartan} \\
x_{i,m+1}^{\pm} x_{j,n}^{\pm} -q^{\pm b_{ij}} x_{i,m}^{\pm} x_{j,n+1}^{\pm} = q^{\pm b_{ij}} x_{j,n}^{\pm} x_{i,m+1}^{\pm} - x_{j,n+1}^{\pm} x_{i,m}^{\pm},  \label{Drinfeld rel: Drinfeld} \\
\sum_{s=0}^{1-c_{ij}} (-1)^s \dstirling{1-c_{ij}}{s}_{q_i}(x_{i,0}^{\pm})^{1-c_{ij}-s} x_{j,0}^{\pm} (x_{i,0}^{\pm})^s = 0 \quad \mathrm{if}\ i \neq j. \label{Drinfeld rel: Serre}
\end{gather}

Let us define three subalgebras of $\qaf$ by generating subsets. The first $U_q^+(\Gaff)$ is generated by all the  $x_{i,m}^+$, the second $U_q^-(\Gaff)$ generated by the $x_{i,m}^-$, and the third $U_q^0(\Gaff)$ by the $\phi_{i,m}^{\pm}$. The algebra $U_q^0(\Gaff)$ is commonly referred to as Drinfeld--Cartan subalgebra. We have a triangular decomposition, that is, a linear isomorphism \cite{Beck,BCP}
\begin{equation}   \label{triangular dec}
    U_q^-(\Gaff)\otimes U_q^0(\Gaff)\otimes U_q^+(\Gaff) \simeq U_q(\Gaff), \quad a \otimes b \otimes c \mapsto abc.
\end{equation}
From Beck's isomorphism \cite{Beck}, we have for $i \in I$ and $m \in \BN$: 
\begin{gather*}
k_i = \phi_{i,0}^+ \in U_q^0(\Gaff),\quad e_i = x_{i,0}^+ \in U_q^+(\Gaff), \quad f_i = x_{i,0}^- \in U_q^-(\Gaff), \\
 k_0^{-1} e_0 \in U_q^-(\Gaff),  \quad x_{i,m}^+,\ x_{i,m+1}^-,\ \phi_{i,m}^+ \in \Borel.
\end{gather*}
With respect to the conjugate action of the Cartan generators $k_i = \phi_{i,0}^+$ for $1\leq i\leq r$, the Hopf algebra $\qaf$ is graded by the root lattice $\BQ$: an element $x \in \qaf$ is of weight $\beta \in \BQ$, and we write $\wt(x) = \beta$, if $k_i x k_i^{-1} = q^{(\alpha_i,\beta)} x$ for $1\leq i \leq r$. The weights of the Drinfeld generators are given by
\begin{gather*}
 \wt(x_{i,m}^{\pm}) = \pm\alpha_i, \quad  \wt(\phi_{i,m}^{\pm}) = 0.
\end{gather*}
Let $\qaf_{\beta}$ denote the subspace of elements in $\qaf$ of weight $\beta$. The weight grading descends to the Borel subalgebra and the above three subalgebras.

\subsection{Shifted quantum affine algebras} 
For $\mu = \sum_{i\in I} m_i \varpi_i^{\vee}$ a coweight, the {\it shifted quantum affine algebra} $\CU_{\mu}(\Gaff)$ is the associative algebra generated by $ x_{i,m}^{\pm}, \phi_{i,m}^{\pm}$ for $(i, m) \in I \times \BZ$
subject to Eqs.\eqref{Drinfel rel: Cartan}--\eqref{Drinfeld rel: Serre} and the following relations \cite{FT}:
\begin{equation}  \label{Drinfeld rel: shift}
\phi_{i,m}^+ = 0 \ \mathrm{if}\ m < 0,\quad \phi_{i,m}^- = 0 \ \mathrm{if}\ m > m_i, \quad \phi_{i,0}^+ \phi_{i,m_i}^- \textrm{ is invertible}.
\end{equation}
For $i \in I$ and $y \in \{x^{\pm}, \phi^{\pm}\}$ let us define the generating series 
$$y_i(z) := \sum_{m\in \BZ} y_{i,m} z^m \in \CU_{\mu}(\Gaff)[[z,z^{-1}]].$$ 
Then $\phi_i^+(z)$ is a power series in $z$ of leading term $\phi_{i,0}^+$, while $\phi_i^-(z)$ is a Laurent series in $z^{-1}$ of leading term $\phi_{i,m_i}^- z^{m_i}$.

The quantum affine algebra $\qaf$ is the quotient of the zero-shifted quantum affine algebra $\CU_0(\Gaff)$ by the ideal generated by the central elements $\phi_{i,0}^+ \phi_{i,0}^- - 1$ for $i \in I$. 

The shifted quantum affine algebra is also graded by the root lattice: an element $x \in \CU_{\mu}(\Gaff)$ is of weight $\beta \in \BQ$ if and only if $\phi_{i,0}^+ x = q^{(\beta,\alpha_i)} x \phi_{i,0}^+$ for all $i \in I$, if and only if $\phi_{i,m_i}^- x = q^{-(\beta,\alpha_i)} x \phi_{i,m_i}^-$ for all $i \in I$. 

Define the three subalgebras $\CU_{\mu}^+(\Gaff),\ \CU_{\mu}^-(\Gaff)$ and $\CU_{\mu}^0(\Gaff)$ of $\CU_{\mu}(\Gaff)$ as in the case of the quantum affine algebra. Then, we have canonical identifications of algebras:
 \begin{equation}  \label{identification subalgebras}
U_q^{\pm}(\Gaff) \cong \CU_{\mu}^{\pm}(\Gaff),\quad x_{i,m}^{\pm} \mapsto x_{i,m}^{\pm}.
 \end{equation}
 Furthermore, a similar triangular decomposition as \eqref{triangular dec} holds: 
 \begin{equation}  \label{triangular dec shifted}
     \CU_{\mu}^-(\Gaff) \otimes \CU_{\mu}^0(\Gaff) \otimes \CU_{\mu}^+(\Gaff) \simeq \CU_{\mu}(\Gaff), \quad a \otimes b \otimes c \mapsto abc.
 \end{equation}

 \begin{prop}\cite[Proposition 3.4]{H}   \label{prop: embedding}
      For an antidominant coweight $\mu$, there is a unique injective algebra homomorphism from the Borel subalgebra $\Borel$ to the shifted quantum affine algebra $\CU_{\mu}(\Gaff)$, denoted by  $\jmath_{\mu}$, which maps $k_0^{-1}e_0 \in U_q^-(\Gaff)$ to the corresponding element of $\CU_{\mu}^-(\Gaff)$ with respect to the identification \eqref{identification subalgebras} and for $m \in \BN$:
 $$ \jmath_{\mu}(x_{i,m}^+) = x_{i,m}^+,\quad \jmath_{\mu}(\phi_{i,m}^+) = \phi_{i,m}^+,\quad \jmath_{\mu}(x_{i,m+1}^-) = x_{i,m+1}^-. $$ 
 \end{prop}

\subsection{Representations of shifted quantum affine algebras} We recall some results on representations of shifted quantum affine algebras from \cite{H}.  

Following \cite[\S 3.1]{HJ}, let $\mathfrak{t}^* := (\BC^{\times})^I$ be the set of $I$-tuples of nonzero complex numbers, with the group structure induced by component-wise multiplication. 
Let $\mu$ be a coweight and $M$ be a module over the shifted quantum affine algebra $\CU_{\mu}(\Gaff)$. For $\lambda  \in \mathfrak{t}^*$ and $i \in I$, let $\lambda(i) \in \BC^{\times}$ denote the $i$th component of $\lambda$ and set
$$ M_{\lambda} := \{v \in M \ |\ \phi_{i,0}^+ v = \lambda(i) v \quad \mathrm{for}\ i \in I \}. $$
If $M_{\lambda}$ is nonzero, then it is called the weight space of weight $\lambda$.  Call $M$ {\it weight graded} if $M$ is a direct sum of the weight spaces. 
\begin{rem}
   Our notion of weight graded modules is weaker than that in \cite[Definition 4.8]{H} since we have no restrictions on actions of the $\phi_{i,\langle \mu, \alpha_i\rangle}^-$; this will also be the case for our Definition \ref{defi: category O} of representation categories below. This does not affect the rationality and classification results in Proposition \ref{prop: rationality}  and Theorem \ref{thm: category O} below. 
\end{rem}
Recall that the root lattice $\BQ$ is a free abelian group generated by the simple roots $\alpha_j$ for $j \in I$. Define $\overline{\alpha_j} := (q^{b_{ji}})_{i\in I} \in \mathfrak{t}^*$. Then, the assignment $\alpha_j\mapsto \overline{\alpha_j}$ extends uniquely to a group homomorphism $\BQ \rightarrow \mathfrak{t}^*,\, \beta \mapsto \overline{\beta}$. It is injective because $q$ is not a root of unity, and it is compatible with the weight grading of representations in the sense that for a module $M$ over a shifted quantum affine algebra $\CU_{\mu}(\Gaff)$:
$$\CU_{\mu}(\Gaff)_{\beta} M_{\lambda} \subset M_{\lambda\overline{\beta}} \quad \textrm{for $\beta \in \BQ$ and $\lambda \in \mathfrak{t}^*$}.$$
We call $M$ top graded if there exists $\lambda \in \mathfrak{t}^*$ such that $M$ is weight graded by $\lambda \overline{\BQ_-}$ and the weight space $M_{\lambda}$ is one-dimensional. Clearly, in this case, $\lambda$ is unique. We refer to $\lambda$ and $M_{\lambda}$  as the top weight and the top weight space. Bottom graded modules are defined in the same way with $\BQ_-$ replaced by $\BQ_+$.

We record the following rationality result from \cite[Proposition 4.10]{H}.
\begin{prop} \cite{H}   \label{prop: rationality}
Let $M$ be a weight graded $\CU_{\mu}(\Gaff)$-module whose weight spaces are finite-dimensional. Then for each weight $\lambda$ of $M$ and for $i \in I$, the image of the generating series $\phi_i^{\pm}(z)$ in $\mathrm{End}(M_{\lambda})((z^{\pm 1}))$ are Laurent expansions, around $z = 0$ and $z = \infty$, respectively, of the same operator-valued rational function of degree $\langle \mu, \alpha_i\rangle$.
\end{prop}
The following definition was made in \cite[Definition 4.8]{H}, mimicking the classical notion of category $\mathcal{O}$ for Kac-Moody algebras.
\begin{defi}\cite{H}  \label{defi: category O}
    Let $\mu$ be a coweight. Define the category $\mathcal{O}_{\mu}$ to be the full subcategory of the category of $\CU_{\mu}(\Gaff)$-modules $M$ satisfying the following properties:  
    \begin{itemize}
        \item[(i)] the $\CU_{\mu}(\Gaff)$-module $M$ is weight graded; 
        \item[(ii)] all weight spaces of $M$ are finite-dimensional; 
        \item[(iii)] there exist $\lambda_1, \lambda_2, \cdots, \lambda_s \in \mathfrak{t}^*$ such that the weights of $M$ are contained in the union of cones $$ (\lambda_1 \overline{\BQ_-}) \cup (\lambda_2 \overline{\BQ_-}) \cup \cdots \cup (\lambda_s \overline{\BQ_-}).$$ 
    \end{itemize}
    Let $\Osh$ be the direct sum of all the categories $\mathcal{O}_\mu$, for $\mu \in \BP^{\vee}$.
\end{defi}

Define $\mathfrak{t}_{\mu}^*$ to be the set of $I$-tuples $(\Bf_i^{\pm}(z))_{i\in I}$ where
$\Bf_i^+(z) \in \BC[[z]]^{\times}$ is an invertible power series in $z$ and $\Bf_i^-(z) \in \BC((z^{-1}))$ is a Laurent series in $z^{-1}$ whose leading term is of degree $\langle \mu, \alpha_i\rangle$ for $i \in I$. Such an $I$-tuple is called an {\it $\ell$-weight} of coweight $\mu$. We also assign the weight $(\Bf_i^+(0))_{i\in I}$ to an $\ell$-weight and define in this way a map $\varpi: \mathfrak{t}_{\mu}^* \rightarrow \mathfrak{t}^*$.
  
Based on the triangular decomposition \eqref{triangular dec shifted} of shifted quantum affine algebras, for $\Bf \in \mathfrak{t}_{\mu}^*$, we define the Verma module $M_{\mu}(\Bf)$ to be the $\CU_{\mu}(\Gaff)$-module generated by $\omega$ subject to the following relations 
$$ \phi_i^{\pm}(z) \omega = \Bf_i^{\pm}(z) \omega,\quad x_i^+(z) \omega = 0 \quad \mathrm{for}\ i \in I. $$
Nonzero quotients of the Verma module $M_{\mu}(\Bf)$ are called modules of highest $\ell$-weight $\Bf$; they are all top graded of top weight $\varpi(\Bf)$. Among them, there is a unique irreducible quotient, denoted by $L_{\mu}(\Bf)$. Lowest $\ell$-weight modules are defined in the same way with $x_i^+(z)$ replaced by $x_i^-(z)$, and they are bottom graded.

We call an $\ell$-weight $\Bf \in \mathfrak{t}_{\mu}^*$ {\it rational} if for $i \in I$:
the series $\Bf_i^+(z)$ and $\Bf_i^-(z)$ are Laurent expansions of the same rational function, denoted by $\Bf_i(z)$, around $z = 0$ and $z = \infty$, respectively.
By abuse of language, we identify $\Bf$ with the $I$-tuple $(\Bf_i(z))_{i\in I}$ of rational functions. We call $\Bf$ {\it polynomial} if it is rational and each $\Bf_i(z)$ is polynomial.  Let $\mathfrak{r}_{\mu}$ be the set of rational $\ell$-weights of coweight $\mu$. Similarly, let $\mathfrak{p}_{\mu}$ be the set of polynomial $\ell$-weights of coweight $\mu$; it is non empty if and only if the coweight $\mu$ is dominant. 
\begin{example} \cite[Definition 3.7]{HJ} \label{defi: prefund}
 For $j \in I$ and $a \in \BC^{\times}$, the $I$-tuple $(1-za \delta_{ij})_{i\in I}$ defines an element of $\mathfrak{p}_{\varpi_j^{\vee}}$, denoted by $\Psi_{j,a}$ and called a prefundamental $\ell$-weight.
\end{example}
We are ready to state the classification result of \cite[Theorem 4.12]{H}.
\begin{theorem}\cite{H} \label{thm: category O}
Let $\mu$ be a coweight. 
\begin{itemize}
    \item [(i)] For $\Bf \in \mathfrak{t}_{\mu}^*$, the irreducible module $L_{\mu}(\Bf)$ is in category $\mathcal{O}_{\mu}$ if and only if $\Bf$ is rational. It is one-dimensional if and only if $\Bf$ is polynomial.

    \item[(ii)] The $L_{\mu}(\Bf)$ for $\Bf \in \mathfrak{r}_{\mu}$ form the set of all mutually non-isomorphic irreducible modules in the category $\mathcal{O}_{\mu}$.
\end{itemize}  
\end{theorem}

The disjoint union of $\mathfrak{t}_{\mu}^*$ over all coweights $\mu$ forms a multiplicative group under component-wise multiplication. We have $\mathfrak{t}_{\mu}^* \mathfrak{t}_{\nu}^* \subset \mathfrak{t}_{\mu+\nu}^*$. The subset of rational $\ell$-weights, namely the disjoint union of $\mathfrak{r}_{\mu}$ over all coweights $\mu$, forms a subgroup denoted by $\mathfrak{r}$. The subset of polynomial $\ell$-weights forms a submonoid denoted by $\mathfrak{p}$. The subset $\mathfrak{p}_0$ of constant $\ell$-weights is a subgroup that is precisely the group $\mathfrak{t}^*$ of $I$-tuples of nonzero complex numbers.

Let $M$ and $N$ be modules over shifted quantum affine algebras $\CU_{\mu}(\Gaff)$ and $\CU_{\nu}(\Gaff)$ respectively. In general, it is unknown how to equip $M\otimes N$ with a $\CU_{\mu+\nu}(\Gaff)$-module structure due to the absence of a shifted Drinfeld--Jimbo coproduct. Nevertheless, there are two situations where this can be done. In each situation, for $i \in I$, the generating series $\phi_i^{\pm}(z)$ of $\CU_{\mu+\nu}(\Gaff)$ act on $M\otimes N$ diagonally as $\phi_i^{\pm}(z) \otimes \phi_i^{\pm}(z)$.

In the first situation, assume that $M$ and $N$ are highest $\ell$-weight modules in categories $\mathcal{O}_{\mu}$ and $\mathcal{O}_{\nu}$ respectively. Then, by the construction of \cite[\S 5]{H}, there is a $\CU_{\mu+\nu}(\Gaff)$-module structure on $M\otimes N$, called {\it fusion product} and denoted by $M \star N$, which is a highest $\ell$-weight module in category $\mathcal{O}_{\mu+\nu}$. The highest $\ell$-weight of $M\star N$ is the product of the highest $\ell$-weights of $M$ and $N$. In particular, for $\Be \in \mathfrak{r}_{\mu}$ and $\Bf \in \mathfrak{r}_{\nu}$, one realizes $L_{\mu+\nu}(\Be\Bf)$ as the unique irreducible quotient of $L_{\mu}(\Be) \star L_{\nu}(\Bf)$. We can iterate the fusion product construction and obtain module structures on tensor products of finitely many highest $\ell$-weight modules in category $\mathcal{O}^{\rm sh}$. 

In the second situation, define $\CU_{\mu}(\Gaff)^{\rm ab}$ to be the quotient of $\CU_{\mu}(\Gaff)$ by its two-sided ideal generated by $x_{i,m}^{\pm}$ for $i\in I$ and $m \in \BZ$. Eq.\eqref{Drinfel rel: Cartan} implies that, for each $i \in I$, the generating series $\phi_i^+(z)$ and $\phi_i^-(z)$ of $\CU_{\mu}(\Gaff)$ have the same image in the quotient. Let $\phi_i(z)$ denote their image. By Eq.\eqref{Drinfeld rel: shift}, we have $\phi_i(z) \in \CU_{\mu}(\Gaff)^{\rm ab}[z]$. We recall also from \cite[Lemma 10.1]{Z} the following algebra homomorphism\footnote{It is $\pi_{\mu}\otimes \mathrm{Id}$ applied to the Drinfeld formal coproduct of \cite[\S 10.1]{FT} where $\pi_{\mu}: \CU_{\mu}(\Gaff) \rightarrow \CU_{\mu}(\Gaff)^{\rm ab}$ is the quotient map. We expect it to be $\pi_{\mu}\otimes \mathrm{Id}$ applied to a conjectural Drinfeld--Jimbo coproduct $\Delta_{\mu,\nu}: \CU_{\mu+\nu}(\Gaff) \rightarrow \CU_{\mu}(\Gaff) \otimes \CU_{\nu}(\Gaff)$, which is true in the case of shifted Yangians \cite[\S 3]{Z}. Similarly, we have another algebra homomorphism $G_{\mu,\nu}: \CU_{\mu+\nu}(\Gaff) \rightarrow \CU_{\mu}(\Gaff) \otimes \CU_{\nu}(\Gaff)^{\rm ab}$. } 
\begin{align} \label{rel: tensor product trivial}
\begin{cases}
     F_{\mu,\nu}: \CU_{\mu+\nu}(\Gaff) \longrightarrow \CU_{\mu}(\Gaff)^{\rm ab} \otimes \CU_{\nu}(\Gaff) \\
     x_i^+(z) \mapsto \phi_i(z) \otimes x_i^+(z),\  x_i^-(z) \mapsto 1 \otimes x_i^-(z), \  \phi_i^{\pm}(z) \mapsto \phi_i(z) \otimes \phi_i^{\pm}(z).
    \end{cases}
\end{align}
We call a $\CU_{\mu}(\Gaff)$-module $M$ {\it trivial} if it factorizes through the quotient map $\pi_{\mu}$, that is, the generators $x_{i,n}^{\pm}$ act identically as zero on $M$. When this is the case, the pullback by $F_{\mu,\nu}$ defines a $\CU_{\mu+\nu}(\Gaff)$-module structure on $M\otimes N$. 

As an example, if $M$ is a one-dimensional $\CU_{\mu}(\Gaff)$-module, then it is trivial. The tensor product $N \mapsto M \otimes N$ defines a functor from the category of $\CU_{\nu}(\Gaff)$-modules to the category of $\CU_{\mu+\nu}(\Gaff)$-modules. We refer to the one-dimensional $\CU_0(\Gaff)$-modules $L_0(\lambda)$ for $\lambda \in \mathfrak{p}_0$ as {\it invertible modules}, in the sense that the tensor product functor $L_0(\lambda) \otimes -$ defines an auto-equivalence of the category of $\CU_{\mu}(\Gaff)$-modules whose inverse is given by $L_0(\lambda^{-1}) \otimes -$. 
If each component of $\lambda$ is $\pm 1$, then $L_0(\lambda)$ is also called a {\it sign} module. 

\subsection{Category $\widehat{\mathcal{O}}$ for the quantum affine algebra}  \label{subsection: O qaf}
Recall in the zero-shifted case the quotient map $\CU_0(\Gaff) \rightarrow \qaf$. Define the category $\widehat{\mathcal{O}}$ to be the full subcategory of $\mathcal{O}_0$ whose objects are $\CU_0(\Gaff)$-modules in category $\mathcal{O}_0$ that factorize through the quotient. This category was first introduced in \cite[\S 4.3]{H1} with the notation $\mathcal{O}(\qaf)$. Here, we adopt the simpler notations in \cite[\S 4]{H}: $\widehat{\mathcal{O}}$ for the quantum affine algebra and $\mathcal{O}$ for the Borel subalgebra to be discussed in the next subsection.

The Drinfeld--Jimbo Hopf algebra structure on $\qaf$ endows the category $\widehat{\mathcal{O}}$ with a monoidal structure. Its simple objects are parameterized by a subset of $\mathfrak{r}_0$ consisting of rational $\ell$-weights $(\Bf_i(z))_{i\in I}$ of degree $0$ such that $\Bf_i(0)\Bf_i(\infty) = 1$ for all $i \in I$. Let $\widehat{\mathfrak{r}}$ denote this subset; it is a subgroup of $\mathfrak{r}$. 

We shall need a distinguished family of irreducible representations in category $\widehat{\mathcal{O}}$.
For $j \in I$, define $\varpi_j \in \mathfrak{p}_0$ to be the $I$-tuple of complex numbers $(q_i^{\delta_{ij}})_{i \in I}$. Consider the following irreducible $\CU_0(\Gaff)$-module in category $\mathcal{O}_0$
\begin{equation}  \label{fund rep}
    W^{(j)} := L_0(\varpi_j \Psi_{j,q_j^2}^{-1} \Psi_{j,1}).
\end{equation}
Indeed, this is a finite-dimensional $\qaf$-module. Let us recall from the classification result of Chari--Pressley \cite{CP,CP1} that finite-dimensional irreducible $\qaf$-modules of type 1 are in bijection with $I$-tuples of polynomials in $(1+z\BC[z])^I$. Our $W^{(j)}$ corresponds to the $I$-tuple $(1-zq_j^2\delta_{ij})_{i\in I}$ in \cite[Theorem 3.3]{CP1} and is commonly called a fundamental representation. In particular, it is also a lowest $\ell$-weight module. Let $\omega_+$ and $\omega_-$ be nonzero vectors of the top weight space and the bottom weight space, respectively. Both are common eigenvectors of the Drinfeld--Cartan series. To describe their eigenvalues, recall the lacing number $r^{\vee}$, the dual Coxeter number $h^{\vee}$ and the involution $i \mapsto \overline{i}$ on $I$.  Then, for $i \in I$, we have
\begin{equation} \label{equ: fundamental repre}
    \phi_i^{\pm}(z) \omega_+ = \left(q_j \frac{1-z}{1-zq_j^2}\right)^{\delta_{ij}} \omega_+,\quad\phi_i^{\pm}(z) \omega_- = \left(q_j^{-1} \frac{1-z q^{r^{\vee}h^{\vee}}q_j^2 }{1-z q^{r^{\vee}h^{\vee}}}\right)^{\delta_{i,\overline{j}}} \omega_-.
\end{equation}
The first equation is by definition, and the second equation is \cite[Lemma 6.8]{FM}.

\subsection{Category $\mathcal{O}$ for the Borel subalgebra}  \label{ss: Borel}
Weight graded modules over the Borel subalgebra $\Borel$ are defined in the usual way based on the joint action of the Cartan generators $k_i = \phi_{i,0}^+$ for $i \in I$. As in Definition \ref{defi: category O}, we have the category $\mathcal{O}$ of $\Borel$-modules, which was previously introduced in \cite[Definition 3.8]{HJ}. The Hopf algebra structure on $\Borel$ makes $\mathcal{O}$ a tensor category.

Consider the triangular decomposition \eqref{triangular dec} of the quantum affine algebra $\qaf$. We take the intersection of each of the three subalgebras at the left-hand side with the Borel subalgebra. This results in three subalgebras $U_q^-(\hat{\mathfrak{b}}),\ U_q^0(\hat{\mathfrak{b}})$ and $U_q^+(\hat{\mathfrak{b}})$ of the Borel algebra. By \cite{Beck}, we still have a linear isomorphism
$$ U_q^-(\hat{\mathfrak{b}}) \otimes U_q^0(\hat{\mathfrak{b}}) \otimes U_q^+(\hat{\mathfrak{b}}) \simeq \Borel,\quad a \otimes b \otimes c \mapsto abc.  $$
Furthermore, the subalgebra $U_q^+(\hat{\mathfrak{b}})$ is generated by $x_{i,m}^+$ for $i \in I$ and $m \in \BN$.
Based on this triangular decomposition, for  $\Bf \in \mathfrak{r}$, we define the Verma module $M(\Bf)$ to be the $\Borel$-module generated by $\omega$ subject to the following relations \cite[\S 3]{HJ}
$$ \phi_i^+(z) \omega = \Bf_i(z) \omega,\quad x_{i,m}^+ \omega = 0 \quad \textrm{for $i \in I$ and $m \in \BN$}. $$
As in the shifted case, the Verma module $M(\Bf)$ is top graded and admits a unique irreducible quotient, denoted by $L(\Bf)$. 

Recall from Proposition \ref{prop: embedding} the embedding of the Borel subalgebra into an anti-dominantly shifted quantum affine algebra. The following result combines \cite[Theorem 3.11]{HJ} with \cite[Corollary 4.11]{H}.
\begin{theorem}\cite{HJ,H}  \label{thm: Borel O}
\begin{itemize}
\item[(i)] The $L(\Bf)$ for $\Bf \in \mathfrak{r}$ form the set of all mutually non-isomorphic irreducible modules in the category $\mathcal{O}$.
\item[(ii)] Let $\mu$ be an antidominant coweight. If $\Bf \in \mathfrak{r}_{\mu}$, then the pullback of the $\CU_{\mu}(\Gaff)$-module $L_{\mu}(\Bf)$ along $\jmath_{\mu}$ is irreducible and isomorphic to the $\Borel$-module $L(\Bf)$.
\end{itemize}
\end{theorem}

The $\Borel$-module $L(\Bf)$ is one-dimensional if and only if  each rational function $\Bf_i(z)$ is a nonzero complex number, equivalently, $\Bf \in \mathfrak{p}_0$. In general, for $\Bf \in \mathfrak{r}$ and $\lambda \in \mathfrak{p}_0$, we have canonical $\Borel$-module isomorphisms
$$ L(\Bf) \otimes L(\lambda) \cong L(\lambda\Bf) \cong L(\lambda) \otimes L(\Bf). $$

\section{Truncated shifted quantum affine algebras}  \label{sec: FT}
In this section, we recall the notion of truncated shifted quantum affine algebras from \cite{FT} and their Jordan--H\"older property from \cite{H}. 

\subsection{A-series and T-series}  \label{ss: A T series}
Fix $\mu = \sum_{i\in I}m_i\varpi^\vee $ to be a coweight. Recall the generating series $\phi_i^{\pm}(z)$ with coefficients in the commutative subalgebra $\CU_{\mu}^0(\Gaff)$ of the shifted quantum affine algebra: $\phi_i^+(z)$ is a power series in $z$ of leading term $\phi_{i,0}^+$, while $\phi_i^-(z)$ is a Laurent series in $z^{-1}$ of leading term $\phi_{i,m_i}^- z^{m_i}$. We normalize them to obtain invertible power series in $z^{\pm 1}$ of constant term 1:
$$ \overline{\phi}_i^+(z) := (\phi_{i,0}^+)^{-1} \phi_i^+(z), \qquad \overline{\phi}_i^-(z) := (\phi_{i,m_i}^-)^{-1} z^{-m_i} \phi_i^-(z).  $$
Define the Drinfeld--Cartan elements $h_{i,s}$, for $i \in I$ and $s \in \BZ_{\neq 0}$, by\footnote{We follow the convention of \cite[(2.2)]{FR1} so that our $h_{i,s}$ corresponds to $[d_i] h_{i,s}$ in \cite{Beck}. This makes our formulas \eqref{def: A}--\eqref{def: T} and notably \eqref{def: R0} much simpler.}
\begin{equation}   \label{def: h}
\overline{\phi}_i^{\pm}(z) =  \exp(\pm (q-q^{-1}) \sum_{\pm s>0} h_{i,s} z^s).
\end{equation}
The $h_{i,s}$ mutually commute and we have for $(i,j,s,m) \in I^2 \times \BZ_{\neq 0} \times \BZ$:
\begin{equation} \label{Drinfeld rel: h x}
[h_{i,s}, x_{j,m}^{\pm}] = \pm \frac{[b_{ij}s]_q}{s} x_{j,m+s}^{\pm} = \pm \frac{[s]_q}{s} B_{ij}(q^s) x_{j,m+s}^{\pm}.
\end{equation}

Recall that the quantum Cartan matrix $B(q)$ is invertible. Let $\widetilde{B}(q^s)$, for $s \in \BZ_{\neq 0}$, denote the inverse of $B(q^s)$. Define the A-series $\BA_i^{\pm}(z)$ and the T-series $\BT_i^{\pm}(z)$, for $i \in I$, to be the following invertible power series in $z^{\pm 1}$ with coefficients in the commutative subalgebra $\CU_{\mu}^0(\Gaff)$ of constant term 1:
\begin{align} 
   \BA_i^{\pm}(z) &:= \exp\left(\pm (q^{-1}-q)\sum_{\pm s > 0} \sum_{j\in I} \widetilde{B}_{ji}(q^s) \frac{1-q_i^{-2s}}{q^s-q^{-s}} h_{j,s} z^s\right),  \label{def: A} \\
   \BT_i^{\pm}(z) &:= \exp\left(\pm (q-q^{-1})\sum_{\pm s > 0} \sum_{j\in I} \widetilde{B}_{ji}(q^s) \frac{1}{q^s-q^{-s}} h_{j,s} z^s\right).  \label{def: T}
\end{align}

\begin{rem} The A-series $\BA_i^-(z)$ appeared first as $ Y_{i,q_i^{-1}}^{-1}q^{2(\rho,\omega_i)} \widetilde{k}_i^{-1}$ in \cite[(3.11)]{FR1}, and the T-series $\BT_i^-(z)$ appeared first as $T_i(z^{-1})$ in \cite[Proposition 5.5]{FH} as a limit of a transfer-matrix. The other halves of these series were introduced in \cite[\S 9.2]{H}: our series $\BA_i^{\pm}(z), \BT_i^+(z)$ and $\BT_i^-(z)$ correspond to the series $Y_i^{\pm}(z)^{-1} \phi_i^{\pm}, T_i^+(z^{-1})^{-1}$ and $T_i^-(z^{-1})$ in {\it loc.cit.}. We modify slightly the definition of \cite{H} in order to make the relations \eqref{rel: A and T}--\eqref{rel: T x-} below more uniform.\end{rem}

We collect the main properties of A-series and T-series.
\begin{prop}\cite{FR1,H,Z}
    In the shifted quantum affine algebra $\CU_{\mu}(\Gaff)$, the following relations hold for $i, j \in I$ and $m \in \BZ$:
    \begin{gather}
    \BA_i^{\pm}(z) = \frac{\BT_i^{\pm}(zq_i^{-2})}{\BT_i^{\pm}(z)},  \label{rel: A and T}  \\
    \overline{\phi}_i^{\pm}(z) = \frac{1}{\BA_i^{\pm}(z)\BA_i^{\pm}(zq_i^2)} \prod_{j\in I: c_{ji}<0} \prod_{t=1}^{-c_{ji}} \BA_j^{\pm}(z q_j^{c_{ji}+2t}), \label{rel: GKLO}   \\
    \BT_i^{\pm}(z) x_{j,m}^- \BT_i^{\pm}(z)^{-1} = x_{j,m}^- - \delta_{ij} x_{i,m\pm 1}^- z^{\pm 1}.  \label{rel: T x-}
\end{gather}
\end{prop}
\begin{proof}
    The first relation is \cite[(9.19)]{H}. The second relation is \cite[(9.18)]{H}, whose negative part appeared first as \cite[(4.9)]{FR1}. The negative part of the third relation follows from \cite[Lemma 9.1]{Z}, whose proof works directly for the positive part.
\end{proof}
We note that Eq.\eqref{rel: GKLO} coincides with \cite[(6.21)]{FT} after a change of variable $z \mapsto z^{-1}$.

In the case of the zero-shift $\mu = 0$, Eqs.\eqref{def: A}--\eqref{def: T} make sense in the quantum affine algebra $\qaf$, which is also a Hopf algebra. Following \cite[Definition 9.2]{Z}, we define the power series $\Theta_i^{\pm}(z) \in \qaf^{\otimes 2}[[z^{\pm 1}]]$ by factorizing the Drinfeld--Jimbo coproduct of the invertible power series $\BT_i^{\pm}(z)$:
\begin{equation}  \label{def: Theta}
    \Delta(\BT_i^{\pm}(z)) = (1\otimes \BT_i^{\pm}(z)) \times\Theta_i^{\pm}(z) \times (\BT_i^{\pm}(z) \otimes 1). 
\end{equation}
We record a polynomiality property for the coproduct of the T-series. 
For $\beta \in \BQ$, the evaluation $\langle \varpi_i^{\vee}, \beta\rangle$ is the coefficient of $\alpha_i$ in $\beta$.

\begin{theorem}\cite{Z}  \label{thm: corpoduct T}
    For $i \in I$, the power series $\Theta_i^{\pm}(z)$ is a sum, over $\beta \in \BQ_+$, of polynomials in $z^{\pm 1}$ with coefficients in $U_q^-(\Gaff)_{-\beta}\otimes U_q^+(\Gaff)_{\beta}$ whose degrees are bounded from above by $\langle \varpi_i^{\vee}, \beta\rangle$. Moreover, $\Theta_{i,0}^{\pm}(z) = 1 \otimes 1$.
\end{theorem}

\begin{proof}
    The negative part follows from \cite[Theorem 9.5]{Z} since our $\BT_i^-(z)$ corresponds to $T_i(z^{-1})$ in {\it loc.cit}. For the positive part, let $U_{q^{-1}}(\Gaff)$ be the quantum affine algebra with deformation parameter $q^{-1}$. We keep the same notation for generators and generating series of this new quantum group. Then the assignments, for $0\leq j \leq r$,
    $$ e_j \mapsto f_j, \quad f_j \mapsto e_j,\quad k_j \mapsto k_j^{-1} $$
    extend uniquely to an algebra anti-isomorphism and co-algebra anti-isomorphism $\Omega: U_{q^{-1}}(\Gaff) \rightarrow \qaf$. By \cite[\S 4.1]{Beck} we know that $\Omega(h_{j,s}) = h_{j,-s}$. This together with Eq.\eqref{def: T} implies that $\Omega(\BT_i^{\pm}(z)) = \BT_i^{\mp}(z^{-1})$. The positive part of the theorem follows from the negative part by applying $\Omega$ to the factorization of \eqref{def: Theta}.
\end{proof}


\subsection{Simply connected truncations and intermediate truncations}  \label{ss: sc truncations}
In this subsection, we study two quotients of shifted quantum affine algebras. They are variations of the original truncated shifted quantum affine algebras \cite[Definition 8.6]{FT} to be recalled in the next subsection. All of these algebras will play a central role in the proof of the main result.

Fix a dominant coweight $\nu = \sum_{i\in I} n_i \varpi_i^{\vee}$ and a polynomial $\ell$-weight $\Bp \in \mathfrak{p}_{\nu}$. Each component $\Bp_i(z)$ is a complex polynomial of degree $n_i$ and has a nonzero constant term. Let $\Bp_{i,n_i}$ and $\Bp_{i,0}$ denote its leading coefficient and its constant term, respectively. 

\begin{rem} These polynomials $\Bp_i(z)$ are sometimes called truncation parameters or flavor parameters \cite{BFN}.
\end{rem}

We attach two new $I$-tuples 
$$\Bp^* = (\Bp^*_i(z))_{i\in I}\text{ and }\Bp^{\sharp} = (\Bp_i^{\sharp}(z))_{i \in I}$$ 
of power series in $z$ and $z^{-1}$, respectively. First, define $\lambda_{i,s} \in \BC$, for $i \in I$ and $s \neq 0$, by the following relations in $\BC[[z]]$ and $\BC[[z^{-1}]]$, respectively:
$$ \Bp_i(z) \Bp_{i,0}^{-1} = \exp\left(\sum_{s>0} \lambda_{i,s} z^s\right),\quad \Bp_i(z) z^{-n_i} \Bp_{i,n_i}^{-1} = \exp\left(\sum_{s<0} \lambda_{i,s} z^s\right).  $$
Then $\Bp_i^*(z) \in \BC[[z]]$ and $\Bp_i^{\sharp}(z) \in \BC[[z^{-1}]]$ are given by (compare with Eq.\eqref{def: A})
\begin{align}     
    \Bp_i^*(z) &:= \exp\left(\sum_{s>0}\sum_{j\in I}  \widetilde{B}_{ji}(q^s) \frac{1-q_i^{-2s}}{q^s-q^{-s}} \lambda_{j,s} z^s\right),   \label{def: p star} \\
    \Bp_i^{\sharp}(z) &:= \exp\left(\sum_{s<0}\sum_{j\in I}  \widetilde{B}_{ji}(q^s) \frac{1-q_i^{-2s}}{q^s-q^{-s}} \lambda_{j,s} z^s\right).
\end{align}
Both power series are of constant term 1. We have the following analog of \eqref{rel: GKLO}:
\begin{align}
    \Bp_i(z) \Bp_{i,0}^{-1} &=   \Bp_i^*(z) \Bp_i^*(zq_i^2) \prod_{j\in I: c_{ji}<0} \prod_{t=1}^{-c_{ji}} \frac{1}{\Bp_j^*(zq_j^{c_{ji}+2t})},  \label{const GKLO+}  \\
     \Bp_i(z) z^{-n_i}\Bp_{i,n_i}^{-1} &=  \Bp_i^{\sharp}(z) \Bp_i^{\sharp}(zq_i^2) \prod_{j\in I: c_{ji}<0} \prod_{t=1}^{-c_{ji}} \frac{1}{\Bp_j^{\sharp}(zq_j^{c_{ji}+2t})}.  \label{const GKLO-} 
\end{align}

Let $\mu = \sum_i m_i \varpi_i^{\vee}$ be another coweight. In the shifted quantum affine algebra $\CU_{\mu}(\Gaff)$ we define, for $i \in I$, the modified A-series as the following invertible power series: 
\begin{equation}   \label{def: SA}
    \mathscr{A}_i^+(z) := \Bp_i^*(z) \BA_i^+(z) \in \CU_{\mu}^0(\Gaff)[[z]], \quad \mathscr{A}_i^-(z) := \Bp_i^{\sharp}(z) \BA_i^-(z) \in \CU_{\mu}^0(\Gaff)[[z^{-1}]].
\end{equation}
It should be noted that both series $\SA_i^{\pm}(z)$ depend implicitly on the polynomial $\ell$-weight $\Bp$. Sometimes we will write $\SA_i^{\Bp,\pm}(z)$ to indicate their dependence. 
We obtain the following relations as a consequence of \eqref{rel: GKLO} and \eqref{const GKLO+}--\eqref{const GKLO-}:
\begin{align}
    \phi_i^+(z) &= \frac{\phi_{i,0}^+\Bp_i(z) \Bp_{i,0}^{-1} }{\SA_i^+(z)\SA_i^+(zq_i^2)} \prod_{j\in I: c_{ji} < 0} \prod_{t=1}^{-c_{ji}} \SA_j^+(zq_j^{c_{ji}+2t}), \label{rel: GKLO+}  \\
    \phi_i^-(z) &= \frac{\phi_{i,m_i}^-\Bp_i(z)  \Bp_{i,n_i}^{-1} z^{m_i-n_i}}{\SA_i^-(z)\SA_i^-(zq_i^2)} \prod_{j\in I: c_{ji} < 0} \prod_{t=1}^{-c_{ji}} \SA_j^-(zq_j^{c_{ji}+2t}). \label{rel: GKLO-}
\end{align}
For $k \in \BZ$, let $\SA_{i,k}^{\pm} \in \CU_{\mu}^0(\Gaff)$ denote the coefficient of $z^k$ in $\SA_i^{\pm}(z)$.
\begin{defi}  \label{defi: truncation}
    The pair $(\mu, \Bp)$ of coweight and polynomial $\ell$-weight is called {\it truncatable} if the unique rational solution to the linear system
    \begin{equation}  \label{truncatable}
        \sum_{j\in I} c_{ji} t_j = n_i - m_i \quad \mathrm{for}\ i \in I
    \end{equation}
    is given by nonnegative integers $t_i \in \BN$. In this situation, define the {\it simply connected truncated shifted quantum affine algebra}  $\,\CU_{\mu}^{\Bp}(\Gaff)$, or {\it simply connected truncation} for short, to be the quotient algebra of $\CU_{\mu}(\Gaff)$ by the following relations for $i \in I$ and $k \in \BZ$:
    \begin{gather}
        \SA_{i,k}^+ = 0 \textrm{\ for $k > t_i$}, \quad    \SA_{i,k}^- = \SA_{i,-t_i}^- \SA_{i,k+t_i}^+ \textrm{\ for $k \in \BZ$},  \label{truncation: poly}  \\
        \Bp_{i,n_i}\phi_{i,0}^+  \prod_{j\in I: c_{ji}<0} (\SA_{j,t_j}^+ q_j^{t_j})^{-c_{ji}} = \Bp_{i,0} \phi_{i,m_i}^- (\SA_{i,t_i}^+ q_i^{t_i})^2.   \label{truncation: simply connected}
    \end{gather}
    Define the {\it intermediate truncated shifted quantum affine algebra} $\overline{\CU}_{\mu}^{\Bp}(\Gaff)$, or {\it intermediate truncation} for short, to be the quotient algebra of $\CU_{\mu}^{\Bp}(\Gaff)$ by the following relations:
    \begin{gather}
         \phi_{i,0}^+ \phi_{i,m_i}^- = \Bp_{i,0}^{-1} \Bp_{i,n_i} \prod_{j\in I} (-q_j)^{-t_j c_{ji}} \textrm{\ for $i \in I$}.  \label{truncation: intermediate}
    \end{gather}
    The resulting two surjective algebra homomorphisms are denoted by\footnote{To simplify notations, we drop the dependence of $\pi_1$ and $\pi_{12}$ on $(\mu, \Bp)$, which will always be clear from the context. The same comment goes for the morphisms $\iota$ in \eqref{adjoint} and $\pi_3$ in \eqref{quotient 3}.}
\begin{equation}  \label{quotient 12}
    \pi_1: \CU_{\mu}(\Gaff) \longrightarrow \CU_{\mu}^{\Bp}(\Gaff), \qquad \pi_{12}: \CU_{\mu}^{\Bp}(\Gaff) \longrightarrow \overline{\CU}_{\mu}^{\Bp}(\Gaff).
\end{equation}
\end{defi}

 Take $k = 0$ in Eq.\eqref{truncation: poly} and notice $\SA_{i,0}^- = 1$. We observe that $\pi_1(\SA_{i,t_i}^+)$ is invertible in $\CU_{\mu}^{\Bp}(\Gaff)$ with inverse $\pi_1(\SA_{i,-t_i}^-)$. 

\begin{lem}  \label{lem: simply connected truncation}
 Let $\mu$ be a coweight, $\Bp$ be a polynomial $\ell$-weight, and $f$ be an algebra homomorphism from $\CU_{\mu}^0(\Gaff)$ to a commutative algebra $\mathcal{S}$. Assume there exists an $I$-tuple $(\SA_i(z))_{i\in I}$ whose components are polynomials in $\mathcal{S}[z]$ of constant term 1 and with invertible leading coefficients, and such that:  
\begin{itemize}
    \item[($\star$)] For $i \in I$,  the Laurent series $f(\phi_i^+(z)) \in \mathcal{S}[[z]]$ and $f(\phi_i^-(z)) \in \mathcal{S}((z^{-1}))$ are Laurent expansions, around $z = 0$ and $z= \infty$ respectively, of precisely the same vector-valued rational function of degree $m_i$: 
$$ \frac{f(\phi_{i,0}^+)\Bp_i(z) \Bp_{i,0}^{-1} }{\SA_i(z) \SA_i(zq_i^2)} \prod_{j\in I: c_{ji} < 0} \prod_{t=1}^{-c_{ji}} \SA_j(z q_j^{c_{ji}+2t}). $$
\end{itemize}
Then, the pair $(\mu,\Bp)$ is truncatable and Eqs.\eqref{truncation: poly}--\eqref{truncation: simply connected} hold in $\mathcal{S}$ after substitutions $\SA_{i,k}^{\pm} \mapsto f(\SA_{i,k}^{\pm}),\ \phi_{i,k}^{\pm} \mapsto f(\phi_{i,k}^{\pm})$ and $t_i \mapsto \deg \SA_i(z)$.
\end{lem}
\begin{proof}
For $k \in \BZ$, let $\SA_{i,k} \in \mathcal{S}$ denote the coefficient of $\SA_i(z)$.
In the factorization of $(\star)$ for $f(\phi_i^-(z)) \in \mathcal{S}((z^{-1}))$, first compare the degrees of the Laurent series in $z^{-1}$:
$$m_i = n_i - 2 \deg \SA_i(z) + \sum_{j\in I: c_{ji}<0} (-c_{ji}) \deg \SA_j(z) = n_i - \sum_{j\in I} c_{ji} \deg \SA_j(z).  $$
This implies that $(\mu, \Bp)$ is truncatable and $t_i$ in Eq.\eqref{truncatable} is $\deg \SA_i(z)$. Then, take the invertible leading coefficients of the Laurent series in $z^{-1}$:
$$ f(\phi_{i,m_i}^-) = \frac{f(\phi_{i,0}^+) \Bp_{i,n_i} \Bp_{i,0}^{-1}}{\SA_{i,t_i}\SA_{i,t_i} q_i^{2t_i}} \prod_{j\in I: c_{ji} < 0} \prod_{t=1}^{-c_{ji}} \SA_{j,t_j} q_j^{(c_{ji}+2t)t_j}.  $$
If we can show $f(\SA_{j,t_j}^+) = \SA_{j,t_j}$, then Eq.\eqref{truncation: simply connected} follows from the invertibility of the leading coefficient $\SA_{j,t_j}$ of $\SA_j(z)$. 

For $i \in I$, let us define 
$$ B_i^+(z) := \SA_i(z)^{-1}f(\SA_i^+(z)) , \quad B_i^-(z) := \SA_i(z)^{-1}  z^{t_i} \SA_{i,t_i}f(\SA_i^-(z)). $$
These are power series in $z^{\pm 1}$ with coefficients in $\mathcal{S}$ of constant term 1.
    Notice that $(\star)$ can be seen as factorizations of $f(\phi_i^{\pm}(z))$ as ratios of Laurent series in $\mathcal{S}((z^{\pm 1}))$ with invertible leading coefficients. Divide all these Laurent series by their leading terms and compare with Eqs.\eqref{rel: GKLO+}--\eqref{rel: GKLO-}. Using the formula above of $f(\phi_{i,m_i}^-)$ in terms of the $\SA_{j,t_j}$, we obtain the following equations for all $i \in I$:   
    $$  1 = \frac{1}{B_i^{\pm}(z) B_i^{\pm}(zq_i^2)} \prod_{j\in I: c_{ji} < 0} \prod_{t=1}^{-c_{ji}} B_j^{\pm}(z q_j^{c_{ji}+2t}) \in 1 + z^{\pm 1} \mathcal{S}[[z^{\pm 1}]].  $$
    From the invertibility of the deformed Cartan matrix we obtain $B_i^{\pm}(z) = 1$, as in the additive case \cite[Lemma 2.1]{GKLO}. It follows that 
$$ f(\SA_{i,k}^+) = \SA_{i,k}, \quad f(\SA_{i,k}^-) = \SA_{i,k+t_i} \SA_{i,t_i}^{-1} = f(\SA_{i,k+t_i}^+) \SA_{i,t_i}^{-1} \quad \textrm{for $k \in \BZ$}. $$
Eq.\eqref{truncation: simply connected} is proved.
Setting $k = -t_i$, we obtain from $\SA_{i,0}^+ = 1$ that
$$f(\SA_{i,-t_i}^-) = f(\SA_{i,0}^+) \SA_{i,t_i}^{-1} = \SA_{i,t_i}^{-1}, \quad f(\SA_{i,k}^-) = f(\SA_{i,k+t_i}^+) f(\SA_{i,-t_i}^-). $$
Eq.\eqref{truncation: poly} is now clear from the polynomiality of $\SA_i(z)$.
\end{proof}

Combining Lemma \ref{lem: simply connected truncation} with Proposition \ref{prop: rationality}, we have the following criterion for a $\CU_{\mu}(\Gaff)$-module structure to factorize through a simply connected truncation.
\begin{cor}  \label{cor: truncation criterion}
    Let $\mu$ be a coweight and $\Bp$ be a polynomial $\ell$-weight. Let $M$ be a weight graded $\,\CU_{\mu}(\Gaff)$-module whose weight spaces are finite-dimensional. Suppose that for each $i \in I$, the image of the generating series $\SA_i^{\Bp,+}(z)$ in $(\mathrm{End} M) [[z]]$ is a polynomial in $z$ whose leading coefficient is invertible. Then:
    \begin{itemize}
        \item[(i)] the pair $(\mu, \Bp)$ is truncatable, and
        \item[(ii)] the module structure on $M$ factorizes through the quotient $\pi_1: \CU_{\mu}(\Gaff) \rightarrow\CU_{\mu}^{\Bp}(\Gaff)$.
    \end{itemize}
\end{cor}
\begin{example}
    Any one-dimensional module descends to a simply connected truncation: if  $\Bp$ is a polynomial $\ell$-weight of coweight $\nu$, then the one-dimensional $\CU_{\nu}(\Gaff)$-module $L_{\nu}(\Bp)$ factorizes through the quotient $\CU_{\nu}^{\Bp}(\Gaff)$. In this case, each $\SA_i^{\pm}(z)$ acts as $1$. This is not true for intermediate truncations due to the additional condition \eqref{truncation: intermediate}.
\end{example}
 Simply connected truncations are compatible with fusion products and tensor products by trivial modules. Let $M$ be a $\CU_{\mu}^{\Bp}(\Gaff)$-module and $N$ be a $\CU_{\nu}^{\mathbf{b}}(\Gaff)$-module. If both modules are of highest $\ell$-weight in category $\Osh$, then their fusion product $M \ast N$ is a $\CU_{\mu+\nu}^{\Bp\mathbf{b}}(\Gaff)$-module. If $M$ is a trivial module, then the tensor product $M \otimes N$ is a $\CU_{\mu+\nu}^{\Bp\mathbf{b}}(\Gaff)$-module. Both statements follow from the diagonal action of $\phi_i^{\pm}(z)$ and, hence, of $\SA_i^{\pm}(z)$ on the tensor space; see \cite[Proposition 12.5]{H}.   Similar compatibility holds for intermediate truncations.

\subsection{Adjoint truncations} We recall the truncated shifted quantum affine algebras originally defined in \cite{FT} and their Jordan--H\"older property from \cite{H}.

Fix $\mu = \sum_{i\in I} m_i \varpi_i^{\vee}$ to be a coweight. The {\it adjoint shifted quantum affine algebra} $\CU_{\mu}^{\mathrm{ad}}(\Gaff)$ is the associative algebra with generators
$$ x_{i,m}^{\pm}, \quad \phi_{i,m}^{\pm}, \quad \psi_i^{\pm} \quad \mathrm{for}\ (i, m) \in I \times \BZ $$
subject to all the defining relations \eqref{Drinfel rel: Cartan}--\eqref{Drinfeld rel: shift}  of the shifted quantum affine algebra $\CU_{\mu}(\Gaff)$ and the following additional relations:
\begin{gather}
    \psi_i^+ \psi_i^- \quad \textrm{is central and invertible}, \quad \phi_{i,0}^+ = \prod_{j\in I} (\psi_j^+)^{c_{ji}}, \quad \phi_{i,m_i}^- = \prod_{j\in I} (\psi_j^-)^{c_{ji}},  \label{rel: psi} \\
    \psi_i^+ \phi_{j,m}^{\pm} = \phi_{j,m}^{\pm} \psi_i^+, \quad
    \psi_i^+ x_{j,m}^{\pm} = q_i^{\pm \delta_{ij}} x_{j,m}^{\pm} \psi_i^+.  \label{rel: psi weight}
\end{gather}
By definition, we have a natural algebra homomorphism 
\begin{equation}  \label{adjoint}
    \iota: \CU_{\mu}(\Gaff) \longrightarrow \CU_{\mu}^{\rm ad}(\Gaff).
\end{equation}
The generating series $\BT_i^{\pm}(z), \,\BA_i^{\pm}(z)$ and $\SA_i^{\pm}(z)$ in Eqs.\eqref{def: A}--\eqref{def: T} and \eqref{def: SA} continue to make sense in $\CU_{\mu}^{\mathrm{ad}}(\Gaff)$ by applying $\iota$. 

Fix $\nu = \sum_{i\in I} n_i \varpi_i^{\vee}$ a dominant coweight and $\Bp$ a polynomial $\ell$-weight of coweight $\nu$ as in Subsection \ref{ss: sc truncations}. Recall, for $i \in I$, the leading coefficient $\Bp_{i,n_i}$ and the constant term $\Bp_{i,0}$ of the polynomial $\Bp_i(z)$. The following definition is essentially due to \cite{FT}.

\begin{defi}\cite{FT}   \label{defi: adjoint truncation}
    Assume the pair $(\mu, \Bp)$ to be truncatable as in Definition \ref{defi: truncation} and let $\mathbf{z} = (\mathbf{z}_i)_{i\in I} \in \mathfrak{p}_0$ be an $I$-tuple of nonzero complex numbers such that
    \begin{equation}  \label{truncation: new parameter}
        \prod_{j\in I} \mathbf{z}_j^{c_{ji}} =  \Bp_{i,0}^{-1} \Bp_{i,n_i} \quad \mathrm{for}\ i \in I. 
    \end{equation}
 The {\it adjoint truncated shifted quantum affine algebra} $\,\CU_{\mu}^{\Bp,\mathbf{z}}(\Gaff)$ associated to this triple $(\mu, \Bp, \mathbf{z})$, or {\it adjoint truncation} for short, is the quotient algebra of $\CU_{\mu}^{\mathrm{ad}}(\Gaff)$ by Eq.\eqref{truncation: poly} and the following relations for $i \in I$:
    \begin{gather}
        \psi_i^+\psi_i^- =  (-q_i)^{-t_i}\mathbf{z}_i, \quad \SA_{i,t_i}^+ = (-1)^{t_i} (\psi_i^+)^2.  \label{truncation: adjoint}
    \end{gather}
    The resulting surjective algebra homomorphism is denoted by
    \begin{equation}  \label{quotient 3}
        \pi_3: \CU_{\mu}^{\rm ad}(\Gaff) \longrightarrow \CU_{\mu}^{\Bp,\mathbf{z}}(\Gaff).
    \end{equation}
\end{defi}
\begin{rem} 
    (i) The A-series in \cite[(10.21)]{H} used to define the truncated shifted quantum affine algebra are related to ours as follows:
    $$ A_i^+(z) := (\psi_i^+)^{-1} \SA_i^+(z), \quad A_i^-(z) := \mathbf{z}_i (\psi_i^-)^{-1} \SA_i^-(z). $$
    Our Definition \ref{defi: adjoint truncation} is a translation of \cite[Definition 10.3]{H} in terms of the $\SA_i^{\pm}(z)$.
    To compare with \cite[Definition 8.6]{FT}, let us specialize the polynomial $Z_i(z)$ of \cite[\S 7.1]{FT} to our $\Bp_i(z^{-1}) \Bp_{i,0}^{-1}$. Then $\phi_i^+$ and $\phi_i^-$ in \cite[\S 7.1]{FT} correspond to our $\psi_i^+$ and $\mathbf{z}_i^{-1}\psi_i^-$. The A-series $A_i^{\pm}(z)$ in \cite[\S 7.2]{FT} are then precisely the above series $A_i^{\pm}(z^{-1})$. The relations in \cite[(8.5)--(8.6)]{FT} correspond to our \eqref{truncation: poly} and \eqref{truncation: adjoint}.

    (ii) Simply connected truncations are compatible with spectral parameter automorphisms. For $c \in \BC^{\times}$ there exists a unique algebra automorphism $\tau_c$ of $\CU_{\mu}(\Gaff)$ that maps any generating series $y_i(z)$ to $y_i(cz)$ for $i \in I$ and $y \in \{x^{\pm}, \phi^{\pm}\}$. If $(\mu, \Bp)$ is truncatable, then $\tau_c$ descends to an algebra isomorphism 
$$\tau_c: \CU_{\mu}^{\Bp}(\Gaff) \longrightarrow \CU_{\mu}^{\mathbf{b}}(\Gaff) $$
where $\mathbf{b} \in \mathfrak{p}$ is the $I$-tuple $(\Bp_i(cz))_{i\in I}$. 
Neither the adjoint truncations nor the intermediate truncation has this compatibility.
\end{rem}
\begin{lem}
    In the situation of Definition \ref{defi: adjoint truncation}, there exists a unique algebra homomorphism $\iota_{23}: \overline{\CU}_{\mu}^{\Bp}(\Gaff) \rightarrow \CU_{\mu}^{\Bp,\mathbf{z}}(\Gaff)$ which fits in the following commutative diagram of algebra homomorphisms
    \begin{gather} \label{diagram}
 \xymatrixcolsep{6pc} \xymatrix{
\CU_{\mu}(\Gaff) \ar[d]_{\iota} \ar[r]^{\pi_1} & \CU_{\mu}^{\Bp}(\Gaff) \ar[r]^{\pi_{12}} & \overline{\CU}_{\mu}^{\Bp}(\Gaff) \ar[dl]^{\iota_{23}}  \\
\CU_{\mu}^{\mathrm{ad}}(\Gaff) \ar[r]^{\pi_3}          & \CU_{\mu}^{\Bp,\mathbf{z}}(\Gaff) }  
\end{gather} 
\end{lem}
\begin{proof}
    We need to verify the defining relations \eqref{truncation: poly}--\eqref{truncation: intermediate} of the intermediate truncation $\overline{\CU}_{\mu}^{\Bp}(\Gaff)$ in the adjoint truncation $\CU_{\mu}^{\Bp,\mathbf{z}}(\Gaff)$. This is clear for \eqref{truncation: poly}. For $i \in I$, we have
    \begin{align*}
        \frac{\psi_i^+}{\psi_i^-} &= \frac{\psi_i^+\psi_i^+}{\psi_i^+\psi_i^-} \xlongequal{\eqref{truncation: adjoint}} \frac{(-1)^{t_i} \SA_{i,t_i}^+}{(-q_i)^{-t_i} \mathbf{z}_i} = \mathbf{z}_i^{-1} \SA_{i,t_i}^+q_i^{t_i},  \\
        \frac{\phi_{i,0}^+}{\phi_{i,m_i}^-} &\xlongequal{\eqref{adjoint}} \prod_{j\in I} (\frac{\psi_j^+}{\psi_j^-})^{c_{ji}} \xlongequal{\eqref{truncation: new parameter}} \prod_{j\in I}\mathbf{z}_j^{-c_{ji}} (\SA_{i,t_i}^+q_i^{t_i})^{c_{ji}} = \Bp_{i,0} \Bp_{i,n_i}^{-1} \prod_{j\in I}  (\SA_{i,t_i}^+q_j^{t_j})^{c_{ji}}, \\
        \phi_{i,0}^+\phi_{i,m_i}^- &= \prod_{j\in I} (\psi_j^+\psi_j^-)^{c_{ji}} = \prod_{j\in I} ((-q_j)^{-t_j}\mathbf{z}_j)^{c_{ji}} = \Bp_{i,0}^{-1} \Bp_{i,n_i}  \prod_{j\in I} (-q_j)^{-t_jc_{ji}}.
    \end{align*}
    The last two lines are the desired relations \eqref{truncation: simply connected}--\eqref{truncation: intermediate}.
\end{proof}

One main advantage of adjoint truncations over simply connected truncations is the following Jordan--H\"older property.
\begin{theorem}\cite{H}  \label{thm: JH}
    Let $(\mu, \Bp, \mathbf{z})$ be as in Definition \ref{defi: adjoint truncation} and $V$ be a $\,\CU_{\mu}(\Gaff)$-module in category $\mathcal{O}_{\mu}$. If $V$ is the pullback of a certain $\,\CU_{\mu}^{\Bp,\mathbf{z}}(\Gaff)$-module along the algebra homomorphism $\pi_3\circ \iota: \CU_{\mu}(\Gaff) \rightarrow \CU_{\mu}^{\Bp,\mathbf{z}}(\Gaff) $, then $V$ is of finite representation length.
\end{theorem}
\begin{proof}  
    Let $\mathcal{O}_{\mu}^{\Bp,\mathbf{z}}$ be the category of $\CU_{\mu}^{\Bp,\mathbf{z}}(\Gaff)$-modules whose pullbacks along $\pi_3\circ \iota$ belong to category $\mathcal{O}_{\mu}$. The pullback functor $(\pi_3\circ\iota)^*: \mathcal{O}_{\mu}^{\Bp,\mathbf{z}} \rightarrow \mathcal{O}_{\mu}$ sends simple objects to simple objects. Therefore, it is enough to show that each object $W$ of category $\mathcal{O}_{\mu}^{\Bp,\mathbf{z}}$ is of finite representation length. 
      By \cite[Theorem 11.15]{H}, there are finitely many simple objects in category $\mathcal{O}_{\mu}^{\Bp,\mathbf{z}}$. Together with the fact that the weights of $W$ are bounded from above and the weight spaces are finite-dimensional, one applies the standard arguments of category $\mathcal{O}$ as in \cite[Chapter 9]{Kac} to conclude.
\end{proof}
The simply connected truncations do not satisfy the Jordan--H\"older property. As a counterexample, the following infinite direct sum of invertible $\CU_0(\Gaff)$-modules
$$ L_0(\overline{\alpha_1}^{-1}) \oplus L_0(\overline{\alpha_1}^{-2}) \oplus L_0(\overline{\alpha_1}^{-3}) \oplus \cdots  $$
is in category $\mathcal{O}_0$, factorizes through $\CU_0^1(\Gaff)$ with $\SA_i^{\pm}(z)$ acting as 1, but is of infinite representation length. 
\subsection{From intermediate truncations to adjoint truncations}

In this subsection, given a module over an intermediate truncation, we explain that under an additional condition, it can be realized as the pullback of a module over an adjoint truncation along the algebra homomorphism $\iota_{23}$ of \eqref{diagram}.

\begin{defi}  \label{defi: balance}
    Let $(\mu, \Bp)$ be a truncatable pair as in Definition \ref{defi: truncation}.
    Call a module $V$ over the intermediate truncation $\overline{\CU}_{\mu}^{\Bp}(\Gaff)$ {\it balanced} if
    \begin{itemize}
        \item[(i)] it is top graded of top weight $\lambda \in \mathfrak{t}^*$;
        \item[(ii)] for $i \in I$, there exists $u_i \in \BC^{\times}$ such that for $\beta \in \BQ_+$, the operator $\SA_{i,t_i}^+$ acts on the weight space $V_{\lambda \overline{\beta}^{-1}}$ as the scalar $u_i q_i^{-2\langle \varpi_i^{\vee},\beta\rangle}$.
    \end{itemize}
\end{defi}
The above balancing condition (ii) can be thought of as a characterization of the weight spaces in terms of the operators $\SA_{i,t_i}^+$. 

\begin{rem} \label{rem: balanced}
From Eqs.\eqref{ss: A T series} and \eqref{rel: T x-} we see that for $k, n \in \BZ$ and $i, j \in I$:
$$  \SA_{i,k-1}^+ x_{j,n}^- - \SA_{i,k}^+ x_{j,n-1}^- = q_i^{-2\delta_{ij}} x_{j,n}^- \SA_{i,n-1}^+ - x_{j,n-1}^- \SA_{i,k}^+ \in \CU_{\mu}(\Gaff). $$
In the simply connected truncation and hence the intermediate truncation, by taking $k = t_i + 1$, this implies the following conjugation relation  
$$\SA_{i,t_i}^+ x_{j,n}^- = q_i^{-2\delta_{ij}} x_{j,n}^- \SA_{i,t_i}^+ \in \overline{\CU}_{\mu}^{\Bp}(\Gaff) \quad \textrm{for $i, j \in I$}. $$
    Let $V$ be a highest $\ell$-weight $\CU_{\mu}(\Gaff)$-module which factorizes through the quotient map $\pi_{12} \circ \pi_1: \CU_{\mu}(\Gaff) \rightarrow \overline{\CU}_{\mu}^{\Bp}(\Gaff)$. Then, the $\overline{\CU}_{\mu}^{\Bp}(\Gaff)$-module $V$ is balanced by taking $u_i$ to be the eigenvalue of $\SA_{i,t_i}^+$ associated to the top weight space of $V$. 
\end{rem}

Recall that a sign module is a one-dimensional $\qaf$-module whose $\ell$-weight is an $I$-tuple of $\pm 1$.
\begin{prop}   \label{prop: balace}
  Let $(\mu, \Bp, \mathbf{z})$ be a triple as in Definition \ref{defi: adjoint truncation} and let $V$ be a balanced module over the intermediate truncation $\overline{\CU}_{\mu}^{\Bp}(\Gaff)$. Then, there exists a sign module $D$ such that the $\overline{\CU}_{\mu}^{\Bp}(\Gaff)$-module $D \otimes V$ is the pullback of a certain $\,\CU_{\mu}^{\Bp,\mathbf{z}}(\Gaff)$-module along the algebra homomorphism $\iota_{23}: \overline{\CU}_{\mu}^{\Bp}(\Gaff) \rightarrow \CU_{\mu}^{\Bp,\mathbf{z}}(\Gaff)$. 
\end{prop}
\begin{proof}
    Combining Eqs.\eqref{truncation: simply connected} and \eqref{truncation: intermediate} we have 
    \begin{equation}  \label{square}
         (\phi_{i,0}^+)^2 = \prod_{j\in I} ((-1)^{t_i} \SA_{i,t_i}^+)^{c_{ji}} \in \overline{\CU}_{\mu}^{\Bp}(\Gaff).
    \end{equation}
    Let $\lambda = (\lambda_i)_{i\in I}$ be the top weight of $V$ and let $(u_i)_{i\in I}$ be as in Definition \ref{defi: balance}. 
    Choose a square root $\kappa_i \in \BC^{\times}$ of $(-1)^{t_i}u_i$ for each $i \in I$. The above relation implies
    \begin{equation}  \label{initial}
        \lambda_i = \varepsilon_i \prod_{j\in I} \kappa_j^{c_{ji}} \quad \textrm{with $\varepsilon_i = \pm 1$}.
    \end{equation}
    By Eq.\eqref{rel: tensor product trivial}, tensor product with the one-dimensional sign module $L_0((\varepsilon_i)_{i\in I})$ does not affect the action of the $\SA_i^{\pm}(z)$ and the module remains balanced. Without loss of generality, we may assume $\varepsilon_i = 1$ for all $i \in I$.

    For $i \in I$, define $\theta_i \in \mathrm{GL}(V)$ by letting its action on each weight space $V_{\lambda\overline{\beta}^{-1}}$ for $\beta \in \BQ_+$ be the nonzero scalar $\kappa_i q_i^{-\langle \varpi_i^{\vee},\beta\rangle}$. We claim that the following relations hold in the $\CU_{\mu}^{\Bp}(\Gaff)$-module $V$ for $i \in I$:
\begin{equation}  \label{equ: auxiliary}
    \SA_{i,t_i}^+ = (-1)^{t_i} \theta_i^2, \qquad \phi_{i,0}^+ = \prod_{j\in I} \theta_j^{c_{ji}}.
\end{equation}
The first relation is a direct consequence of Definition \ref{defi: balance}(ii). For the second relation, by Eq.\eqref{initial} with $\varepsilon_i = 1$ and by definition of weight spaces, both sides act on any weight space as $\lambda_i$ times an integer power of $q$. As $q$ is not a root of unity, it suffices to show that their squares coincide, which is exactly Eq.\eqref{square}.

We show that the $\CU_{\mu}(\Gaff)$-module structure on $V$ can be extended to a $\CU_{\mu}^{\Bp,\mathbf{z}}(\Gaff)$-module structure by setting
$$ \psi_i^+ \mapsto \theta_i,\quad \psi_i^- \mapsto \mathbf{z}_i (-q_i)^{-t_i} \theta_i^{-1}. $$
It suffices to prove the relations \eqref{rel: psi}, \eqref{rel: psi weight} and \eqref{truncation: adjoint}. The first half of \eqref{truncation: adjoint} follows from the definition. The second half of \eqref{truncation: adjoint} and the first two relations of \eqref{rel: psi} follow directly from Eq.\eqref{equ: auxiliary}. Since the $\theta_i$ are defined using the weight grading, they mutually commute and by our choice of $\kappa_i$ we get Eq.\eqref{rel: psi weight}. The third relation of \eqref{rel: psi}
follows from Eqs.\eqref{truncation: intermediate} and \eqref{truncation: new parameter}.
\end{proof}
The intermediate truncation also satisfies the Jordan--Hölder property.
 \begin{cor}
     Let $(\mu, \Bp)$ be a truncatable pair. Up to isomorphism, there are finitely many simple $\overline{\CU}_{\mu}^{\Bp}(\Gaff)$-modules in category $\mathcal{O}_{\mu}$. Each $\overline{\CU}_{\mu}^{\Bp}(\Gaff)$-module in category $\mathcal{O}_{\mu}$ is of finite representation length.
 \end{cor}
 \begin{proof}
Any simple $\overline{\CU}_{\mu}^{\Bp}(\Gaff)$-module $V$
in category $\mathcal{O}_{\mu}$ is of highest $\ell$-weight and hence balanced by Remark \ref{rem: balanced}. By Proposition \ref{prop: balace}, up to tensor product with a sign module, it is the pullback of a simple module $W$ in category $\mathcal{O}_{\mu}^{\Bp,\mathbf{z}}$. The finiteness of $V$ follows from the finiteness of $W$ (as shown in \cite[Theorem 11.15]{H} and recalled in the proof of Theorem \ref{thm: JH}) and that of sign modules. The second statement is proved as in Theorem \ref{thm: JH} by replacing the adjoint truncation $\CU_{\mu}^{\Bp,\mathbf{z}}(\Gaff)$ with the intermediate truncation $\overline{\CU}_{\mu}^{\Bp}(\Gaff)$.
 \end{proof}

\section{Truncation series from R-matrices}\label{TruncRmat}  
In this section, we reconstruct the A-series $\BA_i^+(z)$ of Eq.\eqref{def: A} from suitable evaluations of the universal R-matrix of the quantum affine algebra. Then we deduce from a previous work (Theorem \ref{thm: poly R Borel} established in \cite{Z}) the polynomiality of A-series acting on a tensor product of several irreducible modules in category $\mathcal{O}$ for the Borel subalgebra (Theorem \ref{thm: polynomiality A series}).

\subsection{The Universal R-matrix}

The quantum affine algebra $\qaf$ is quasi-triangular as a Hopf algebra. 
We shall need the reduced part of the universal R-matrix. Recall the inverse $(\widetilde{B}_{ij}(q))_{i,j\in I}$ of the symmetric quantum Cartan matrix $([b_{ij}]_q)_{i,j\in I}$ and the Drinfeld--Cartan elements $h_{i,s}$ of Eq.\eqref{def: h} viewed in the quantum affine algebra. The {\it abelian part} of the universal R-matrix is a power series in $z$ with coefficients in $\Borel \otimes \qaf$ defined by \cite{Damiani} 
\begin{equation}
     \CR_0(z) := \exp \left(\sum_{j,k\in I}\sum_{s>0} \frac{ s(q^{-1}-q)\widetilde{B}_{jk}(q^s)}{[s]_q} h_{j,s} \otimes h_{k,-s} z^s \right). \label{def: R0} 
\end{equation}
The triangular parts $\CR_{\pm}(z)$ are sums, over the positive root cone $\BQ_+$, of power series $\CR_{\pm}^{\beta}(z)$ in the form
\begin{equation}  
       \CR_{\pm}^0(z) = 1 \otimes 1,\quad \CR_{\pm}^{\beta}(z) \in (\Borel_{\pm \beta} \otimes \qaf_{\mp \beta})[[z]]. \label{def: R+-}
\end{equation}
The {\it reduced part} of the universal R-matrix is given by the factorization \cite{Damiani}:
\begin{equation*}
\barR(z) := \CR_+(z) \CR_0(z) \CR_-(z). 
\end{equation*}
It is a power series in $z$ with coefficients in a suitable completion of the tensor product algebra $\Borel \otimes \qaf$. At the level of representations the completion is inessential: given an arbitrary $\Borel$-module $V$ and a finite-dimensional $\qaf$-module $W$, we can evaluate $\barR(z)$ at $V \otimes W$ to obtain a linear operator \cite{FR1}
$$ \barR_{V,W}(z): V \otimes W \longrightarrow (V \otimes W)[[z]]. $$
As a particular example, we take $W$ as a fundamental representation of \eqref{fund rep}.
The following result is proved implicitly in \cite[\S 3.3]{FR1}. See also \cite[Lemma 2.6]{FM} for a close situation. We view the positive series $\BA_i^+(z)$ of Eq.\eqref{def: A} as a power series in $z$ with coefficients in $\Borel$ so that it acts on the $\Borel$-module $V$.
\begin{prop}  \label{prop: A series from R}
    Let $i \in I$ and $V$ be a $\Borel$-module. In the fundamental module $W^{(i)}$, fix $\omega$ to be a nonzero vector of top weight $\varpi_i$. Then for $v \in V$, we have 
    $$ \barR_{V, W^{(i)}}(z) (v \otimes \omega) \equiv \BA_i^+(z)v  \otimes \omega \ \mathrm{mod}.\ \left(\sum_{\beta \in \BQ_+\setminus\{0\}} V \otimes W^{(i)}_{\varpi_i\overline{\beta}^{-1}}\right)[[z]]. $$
\end{prop}
\begin{proof}
    From the weight decomposition of Eq.\eqref{def: R+-} only the term $\beta = 0$ contributes to $\CR_-(z)(v\otimes \omega)$. For $k\in I$ and $0\neq s \in \BZ$, observe from Eqs.\eqref{equ: fundamental repre} and \eqref{def: h} that the eigenvalue of $h_{k,-s}$ associated to $\omega$ is $\delta_{ik} \frac{q_i^{-2s}-1}{s(q^{-1}-q)}$. In Eq.\eqref{def: R0}, by replacing each $h_{k,-s}$ with its eigenvalue, we recover $\BA_i^+(z)$ in \eqref{def: A}. This gives
    $$\CR_0(z)\CR_-(z) (v \otimes \omega) = \BA_i^+(z) v \otimes \omega$$ 
   For $\beta \in \BQ_+\setminus \{0\}$, the $(\beta,-\beta)$-component $\CR_+^{\beta}(z)$ applied to $\BA_i^+(z) v \otimes \omega$ becomes a power series in $z$ with coefficients in $V \otimes W_{\varpi_i\overline{\beta}^{-1}}^{(i)}$. This proves the proposition.
\end{proof}

\subsection{Polynomiality of R-matrices}

Recall from Definition \ref{defi: prefund} the prefundamental $\ell$-weights $\Psi_{i,a}$ for $i \in I$ and $a \in \BC^{\times}$. Each polynomial $\ell$-weight $\Bf \in \mathfrak{p}$ factorizes uniquely as the product of a constant $\ell$-weight and a monomial in the prefundamental $\ell$-weights:
$$\Bf = \lambda \Psi_{i_1,a_1} \Psi_{i_2,a_2} \cdots \Psi_{i_n,a_n} \quad \mathrm{with}\ \lambda \in \mathfrak{p}_0.  $$
We associate to $\Bf$ its T-series as a generalization of Eq.\eqref{def: T}:
$$ T_{\Bf}(z) := \BT_{i_1}^-(z^{-1}a_1^{-1}) \BT_{i_2}^-(z^{-1}a_2^{-1}) \cdots \BT_{i_n}^-(z^{-1}a_n^{-1}) \in \qaf[[z]]. $$
For $i \in I$, let $t_{\Bf,i}^+(z)$ (resp. $t_{\Bf,i}^-(z)$) denote the eigenvalue of $T_{\Bf}(z)$ acting on the top weight space (resp. the bottom weight space) of the $i$th fundamental representation $W^{(i)}$ of $\qaf$. Both are invertible power series in $z$ of constant term 1.
The following result is \cite[Theorem 11.4]{Z} applied to $W^{(i)}$. 
 \begin{theorem}\cite{Z}  \label{thm: poly R Borel}
 Let $s \geq 1$ and $\Bm_k, \Bn_k \in \mathfrak{p}$ be polynomial $\ell$-weights for $1\leq k \leq s$. Set $V$ to be the tensor product of $\Borel$-modules 
 $$ V := L\left(\frac{\Bm_1}{\Bn_1}\right) \otimes L\left(\frac{\Bm_2}{\Bn_2}\right) \otimes \cdots \otimes L\left(\frac{\Bm_s}{\Bn_s}\right) $$ 
in the tensor category $\mathcal{O}$ of Subsection \ref{ss: Borel}. For $i \in I$, define the power series
 \begin{align*}
 \alpha_{V,i}(z) :=  \frac{t_{\Bn_1,i}^-(zq^{2r^{\vee}h^{\vee}})}{t_{\Bm_1,i}^+(z)}\frac{t_{\Bn_2,i}^-(zq^{2r^{\vee}h^{\vee}})}{t_{\Bm_2,i}^+(z)} \cdots \frac{t_{\Bn_s,i}^-(zq^{2r^{\vee}h^{\vee}})}{t_{\Bm_s,i}^+(z)} \in \BC[[z]].
 \end{align*}
 Then the linear operator $\alpha_{V,i}(z) \barR_{V,W^{(i)}}(z)$ sends $V \otimes W^{(i)}$ to $V \otimes W^{(i)}[z]$. 
 \end{theorem}
 \begin{proof}
     In the case where all the $\Bm_k$ and $\Bn_k$ are monomials of prefundamental $\ell$-weights, this is \cite[Theorem 11.4]{Z}. In general, factorize $\Bm_k = \lambda_k \Bm_k'$ and $\Bn_k = \mu_k \Bn_k'$ with $\lambda_k, \mu_k \in \mathfrak{p}_0$ constant and $\Bm_k', \Bn_k'$ monomials of the prefundamental $\ell$-weights. Define the tensor product module $V'$ in a similar way using the $\Bm_k'$ and $\Bn_k'$. Then $V \cong D \otimes V'$ as $\Borel$-modules where $D$ is one-dimensional of highest $\ell$-weight 
     $$\lambda_1\lambda_2 \cdots \lambda_s (\mu_1\mu_2\cdots \mu_s)^{-1} \in \mathfrak{p}_0. $$
     Since the Drinfeld--Cartan elements $h_{j,s}$ act as zero on $D$, their action on $V$ and $V'$ are the same if we identify the underlying vector spaces. This gives $\alpha_{V',i}(z) = \alpha_{V,i}(z)$. The polynomiality of $\overline{\CR}_{V,W^{(i)}}(z)$ follows from that of  $\overline{\CR}_{V',W^{(i)}}(z)$. 
 \end{proof}

\subsection{Local A-polynomiality}
 Recall that in Eq.\eqref{def: p star} we assign to a polynomial $\ell$-weight $\Bp \in \mathfrak{p}$ an $I$-tuple $\Bp^*$ of power series in $z$ with constant term 1. Recall also the involution $i\mapsto \overline{i}$ on $I$.

The following result computes the eigenvalues of T-series associated to the top space and the bottom space of a fundamental module in terms of the map $\Bp \mapsto \Bp^*$.

\begin{lem}   \label{lem: truncation poly}
    Let $\Bf \in \mathfrak{p}$ be a polynomial $\ell$-weight factorized as follows:
    $$\Bf = \lambda \Psi_{i_1,a_1} \Psi_{i_2,a_2} \cdots \Psi_{i_n,a_n} \quad \mathrm{with}\ \lambda \in \mathfrak{p}_0.   $$
    Then as $I$-tuples of power series in $z$ we have 
    \begin{equation*}
        \left(\frac{1}{t_{\Bf,i}^+(z)}\right)_{i\in I} = \Bf^*,\quad (t_{\Bf,i}^-(z))_{i\in I} = (\Psi_{\overline{i_1},a_1q^{-r^{\vee}h^{\vee}}} \Psi_{\overline{i_2},a_2q^{-r^{\vee}h^{\vee}}} \cdots \Psi_{\overline{i_n},a_nq^{-r^{\vee}h^{\vee}}})^*.
    \end{equation*}
\end{lem}
\begin{proof}
    By definition, both sides are multiplicative with respect to $\Bf$ and are independent of the constant term $\lambda$. It suffices to consider the case $\Bf = \Psi_{k,a}$ for $k \in I$ and $a\in \BC^{\times}$. Fix $i \in I$. We need to compute the eigenvalues of $T_{\Bf}(z) = \BT_k^-(z^{-1}a^{-1})$ associated to the top weight space and the bottom weight space of the fundamental module $W^{(i)}$. 

    Let $\omega_+$ and $\omega_-$ denote a highest $\ell$-weight vector and a lowest $\ell$-weight vector of $W^{(i)}$. Set $c := q^{r^{\vee}h^{\vee}}$. By Eqs.\eqref{equ: fundamental repre} and \eqref{def: h}, we have for $j \in I$ and $0 \neq s \in \BZ$:
    \begin{equation*}   
   h_{j,s} \omega_+ = \delta_{ij} \frac{q_i^{2s}-1}{s(q-q^{-1})} \omega_+, \quad h_{j,s} \omega_- = \delta_{j,\overline{i}} \frac{c^s(1-q_i^{2s}) }{s(q-q^{-1})} \omega_-.
\end{equation*}
    In Eq.\eqref{def: T} for $\BT_k^-(z^{-1}a^{-1})$, replace each $h_{j,s}$ by its eigenvalues associated to $\omega_{\pm}$. We get the corresponding eigenvalues of $T_{\Bf}(z) $  as follows (notice that $\widetilde{B}(q) = \widetilde{B}(q^{-1})$):
    \begin{align*}
    t_{\Bf,i}^+(z) &=  \exp\left(\sum_{s > 0} \sum_{j\in I} \widetilde{B}_{jk}(q^s) \frac{q^{-1}-q}{q^{-s}-q^s} \delta_{ij} \frac{q_i^{-2s}-1 }{s(q^{-1}-q)}  a^sz^s\right)  \\
        &= \exp\left(\sum_{s > 0} \widetilde{B}_{i k}(q^s)  \frac{q_i^{-2s}-1}{s(q^{-s}-q^s)}  a^sz^s\right) = \exp\left(\sum_{s > 0} \widetilde{B}_{ki}(q^s)  \frac{q_i^{-2s}-1}{s(q^{-s}-q^s)}  a^sz^s\right) \\
        &= \exp\left(\sum_{s > 0} \sum_{j\in I}\widetilde{B}_{j i}(q^s)  \frac{1-q_i^{-2s}}{q^{s}-q^{-s}}\delta_{jk} \frac{ a^s}{s} z^s\right), \\
        t_{\Bf,i}^-(z) &= \exp\left(\sum_{s > 0} \sum_{j\in I} \widetilde{B}_{jk}(q^s) \frac{q^{-1}-q}{q^{-s}-q^s} \delta_{j,\overline{i}} \frac{c^{-s} (1-q_i^{-2s}) }{s(q^{-1}-q)}  a^sz^s\right)   \\
        &= \exp\left(\sum_{s > 0} \widetilde{B}_{\overline{i}, k}(q^s)  \frac{c^{-s} (1-q_i^{-2s}) }{s(q^{-s}-q^s)}  a^sz^s\right) = \exp\left(\sum_{s > 0} \widetilde{B}_{\overline{k}, i}(q^s)  \frac{c^{-s} (1-q_i^{-2s}) }{s(q^{-s}-q^s)}  a^sz^s\right)\\
        &= \exp\left(\sum_{s > 0} \sum_{j\in I} \widetilde{B}_{j i}(q^s)  \frac{1-q_i^{-2s} }{q^s-q^{-s}} \delta_{j,\overline{k}} \frac{-c^{-s}a^s}{s} z^s\right).
    \end{align*}
    In the second line and the fifth line, we have used the relations $\widetilde{B}_{ij} = \widetilde{B}_{ji} = \widetilde{B}_{\overline{i},\overline{j}}$ and $\overline{\overline{i}} = i$ for $i, j \in I$.
   It follows from Eq.\eqref{def: p star} and the relation
    $$\exp\left(-\sum_{s>0} \frac{a^s}{s} z^s\right) = 1- a z  $$
   that $t_{\Bf,i}^+(z)^{-1}$ and $t_{\Bf,i}^-(z)$ are the $i$th components of $\Psi_{k,a}^*$ and $\Psi_{\overline{k},c^{-1}a}^*$, respectively.
\end{proof}

Let $\Bn \mapsto \widetilde{\Bn}$ denote the monoid automorphism of $\mathfrak{p}$ which fixes the constant $\ell$-weights and sends $\Psi_{i,a}$ to $\Psi_{\overline{i},aq^{r^{\vee}h^{\vee}}}$ for $i \in I$ and $a \in \BC^{\times}$. The main result of this section is the following local polynomiality property of A-series.
\begin{theorem}  \label{thm: polynomiality A series}
    Let $s \geq 1$ and $\Bm_k, \Bn_k \in \mathfrak{p}$ be polynomial $\ell$-weights for $1\leq k \leq s$. Set $V$ to be the tensor product of $\Borel$-modules
 $$ V := L\left(\frac{\Bm_1}{\Bn_1}\right) \otimes L\left(\frac{\Bm_2}{\Bn_2}\right) \otimes \cdots \otimes L\left(\frac{\Bm_s}{\Bn_s}\right). $$ 
Define the following polynomial $\ell$-weight 
 \begin{align*}
 \Bp :=  \Bm_1 \Bm_2 \cdots \Bm_s \widetilde{\Bn_1} \widetilde{\Bn_2} \cdots \widetilde{\Bn_s}.
 \end{align*}
 Then, for $i \in I$, the generating series $\Bp_i^*(z) \BA_i^+(z) \in \Borel[[z]]$ sends $V$ to $V[z]$.
\end{theorem}
\begin{proof}
    By Lemma \ref{lem: truncation poly} we can identify the power series $\alpha_{V,i}(z)$ of Theorem \ref{thm: poly R Borel} with $\Bp_i^*(z)$. Then apply the polynomiality of Theorem \ref{thm: poly R Borel} to Proposition \ref{prop: A series from R}.
\end{proof}
The above theorem does not imply the {\it global} polynomiality of $\SA_i^+(z) = \Bp_i^*(z) \BA_i^+(z)$ acting on $V$: the degrees of $\SA_i^+(z)$ restricted to weight spaces may not be upper bounded. This is more complicated than the Yangian situation, for which the corresponding A-series is a Laurent series in $z^{-1}$ with fixed leading term $z^{t_i}$ so that local polynomiality \cite[Theorem 8.4]{HZ} implies already global polynomiality.

\section{Invertibility of leading coefficients  and descent to truncations}\label{invdom}
In this section, we prove one of the main results of this paper: any irreducible module in category $\mathcal{O}_{\mu}$ factorizes through a truncated shifted quantum affine algebra $\CU_{\mu}^{\Bp}(\Gaff)$ for an explicit truncation parameter $\Bp$ (Theorem  \ref{thm: truncation shifted}). A similar statement is established for a tensor product of several irreducible modules in category $\widehat{\mathcal{O}}$. As an intermediate result, we give in Proposition \ref{prop: key} a sufficient condition for a module in category $\Osh$ to descend to a truncation. In our proofs, a key step is the invertibility of the leading coefficients for certain operator-valued polynomials.

\subsection{T-polynomiality}

\begin{defi}  \label{defi: f}
Let $\mu$ be a coweight and $V$ be a top graded module over the shifted quantum affine algebra $\CU_{\mu}(\Gaff)$ of top weight $\lambda \in \mathfrak{t}^*$. For $i \in I$, define $f_i^V(z) \in \BC[[z]]$ to be the eigenvalue of $\BT_i^+(z) \in \CU_{\mu}(\Gaff)[[z]]$ acting on $V_{\lambda}$ and set
$$ \overline{\BT}_i^V(z) := f_i^V(z)^{-1} \BT_i^+(z)|_V \in (\mathrm{End}V)[[z]].  $$
We say that $V$ is {\it T-polynomial} if for all $i\in I$: restricted to each weight space $V_{\lambda\overline{\beta}^{-1}}$ with $\beta \in \BQ_+$, the operator $\overline{\BT}_i^V(z)$ is a polynomial in $z$ of degree $\langle \varpi_i^{\vee}, \beta\rangle$ whose leading coefficient is invertible. 
\end{defi}
In the above definition, we do not require $V$ to be in category $\mathcal{O}_{\mu}$. The definition also makes sense for $\Borel$-modules by viewing $\BT_i^+(z)$ as a power series with coefficients in $\Borel$.
Our main examples of T-polynomial modules are as follows.
\begin{prop}  \label{prop: poly T} 
\begin{itemize}
    \item[(i)] Any highest $\ell$-weight module over a shifted quantum affine algebra is T-polynomial.
    \item[(ii)] Any tensor product of T-polynomial $\qaf$-modules is T-polynomial.
\end{itemize}
\end{prop}
\begin{proof}
The idea is the same as that of \cite[Proposition 6.2]{Z}, based on Eqs.\eqref{rel: T x-}--\eqref{def: Theta} and Theorem \ref{thm: corpoduct T}. Let $V$ be a $\CU_{\mu}(\Gaff)$ of top weight $\lambda$. For $i \in I$ and $\beta \in \BQ_+$, let $D_{i,\beta}^V \in \mathrm{End}(V_{\lambda \overline{\beta}^{-1}})$ denote the coefficient of $z^{\langle \varpi_i^{\vee},\beta\rangle}$ in the power series $\overline{\BT}_i^V(z)|_{V_{\lambda \overline{\beta}^{-1}}}$. 

(i) Assume $V$ is generated by a highest $\ell$-weight vector $\omega$. Then, for each $\beta \in \BQ_+$, the weight space $V_{\lambda\overline{\beta}^{-1}}$ is spanned by the $x_{i_1,m_1}^- x_{i_2,m_2}^- \cdots x_{i_s,m_s}^- \omega$ such that $\beta = \alpha_{i_1}+\alpha_{i_2} + \cdots + \alpha_{i_s}$. Applying $\overline{\BT}_i^V(z)$ to such a vector and using Eq.\eqref{rel: T x-}:
$$ \overline{\BT}_i^V(z) x_{j,m}^- = (x_{j,m}^- - z \delta_{ij} x_{j,m+1}^-) \overline{\BT}_i^V(z) $$
we find that $\overline{\BT}_i^V(z)|_{V_{\lambda \overline{\beta}^{-1}}}$ is a polynomial of degree bounded from above by the number of $i$'s in the sequence $i_1i_2\cdots i_s$, namely $\langle \varpi_i^{\vee},\beta\rangle$. Its coefficient of $z^{\langle \varpi_i^{\vee},\beta\rangle}$ acts as
$$ D_{i,\beta}^V:   x_{i_1,m_1}^- x_{i_2,m_2}^- \cdots x_{i_s,m_s}^- \omega \mapsto (-1)^{\langle \varpi_i^{\vee},\beta\rangle}x_{i_1,m_1+\delta_{i_1i}}^- x_{i_2,m_2+\delta_{i_2i}}^- \cdots x_{i_s,m_s+\delta_{i_si}}^- \omega. $$
It suffices to prove the invertibility of $D_{i,\beta}^V$. Let $h_i^V(z) \in \BC[[z^{-1}]]$ be the eigenvalue of $\BT_i^-(z) \in \CU_{\mu}(\Gaff)[[z^{-1}]]$ acting on the top weight space of $V$. Then, from the negative part of Eq.\eqref{rel: T x-} we have a similar polynomiality of $h_i^V(z)^{-1} \BT_i^-(z)$ acting on each weight space $V_{\lambda\overline{\beta}^{-1}}$. Its coefficient of $z^{\langle \varpi_i^{\vee},\beta\rangle}$ acts as
$$ x_{i_1,m_1}^- x_{i_2,m_2}^- \cdots x_{i_s,m_s}^- \omega \mapsto (-1)^{\langle \varpi_i^{\vee},\beta\rangle}x_{i_1,m_1-\delta_{i_1i}}^- x_{i_2,m_2-\delta_{i_2i}}^- \cdots x_{i_s,m_s-\delta_{i_si}}^- \omega. $$
This is exactly the inverse of $D_{i,\beta}^V$.

(ii) Let $V$ and $W$ be T-polynomial $\qaf$-modules with top weights $\lambda_1$ and $\lambda_2$ respectively. From the unitriangularity of Theorem \ref{thm: corpoduct T} we see that the eigenvalue $f_i^{V\otimes W}(z)$ of $\BT_i^+(z)$ associated to the top weight space of $V\otimes W$ is $f_i^V(z)f_i^W(z)$. From now on, fix $\gamma \in \BQ_+$ and $v \in (V\otimes W)_{\lambda_1\lambda_2\overline{\gamma}^{-1}}$. By Eq.\eqref{def: Theta}, we have 
$$\overline{\BT}_i^{V\otimes W}(z)  v = \sum_{\beta\in \BQ_+} p_{\beta}(z) \quad \textrm{where } p_{\beta}(z) := (1\otimes \overline{\BT}_i^W(z)) \Theta_{i,\beta}^+(z) (\overline{\BT}_i^V(z) \otimes 1)  v.$$
For weight reason, all but finitely many $p_{\beta}(z)$ is zero.

By linearity, we may assume $v = v_1 \otimes v_2 \neq 0$ where $v_1 \in V_{\lambda_1\overline{\gamma_1}^{-1}},\, v_2 \in W_{\lambda_2\overline{\gamma_2}^{-1}}$ and $\gamma_1, \gamma_2 \in \BQ_+$ such that $\gamma = \gamma_1+\gamma_2$. Then, from $\Theta_{i,0}^+(z) = 1 \otimes 1$ and the T-polynomiality of $V$ and $W$, we obtain that $p_0(z)$ is a polynomial in $z$ of leading term $z^{\langle \varpi_i^{\vee},\gamma_1+\gamma_2 \rangle}D_{i,\gamma_1}^V(v_1) \otimes D_{i,\gamma_2}^W(v_2) $.   
For $\beta \in \BQ_+\setminus \{0\}$, we have the following estimations.
\begin{itemize}
    \item[(a)] By T-polynomiality of $V$, the power series $\overline{\BT}_i^V(z) v_1 \otimes v_2$ is a polynomial in $z$ with coefficients in $V_{\lambda_1\overline{\gamma_1}^{-1}} \otimes W_{\lambda_2\overline{\gamma_2}^{-1}}$ of degree $\langle\varpi_i^{\vee},\gamma_1\rangle$.
    \item[(b)] Apply the polynomial $\Theta_{i,\beta}^+(z)$ of Theorem \ref{thm: corpoduct T} to the polynomial $\overline{\BT}_i^V(z) v_1 \otimes v_2$ of (a). We get a polynomial with coefficients in $V_{\lambda_1\overline{\gamma_1}^{-1}\overline{\beta}^{-1}} \otimes W_{\lambda_2\overline{\gamma_2}^{-1}\overline{\beta}}$ of degree bounded from above by 
    $$ \langle \varpi_i^{\vee}, \gamma_1\rangle + \deg \Theta_{i,\beta}^+(z) \leq \langle \varpi_i^{\vee}, \gamma_1\rangle+\langle \varpi_i^{\vee}, \beta\rangle = \langle \varpi_i^{\vee}, \gamma_1+\beta\rangle. $$
    \item[(c)] Apply $1\otimes \overline{\BT}_i^W(z)$ to the polynomial $\Theta_{i,\beta}^+(z)(\BT_i^V(z)v_1\otimes v_2)$ of (b). We obtain from the T-polynomiality of $W$ that $p_{\beta}(z)$ is a polynomial with coefficients in $V_{\lambda_1\overline{\gamma_1}^{-1}\overline{\beta}^{-1}} \otimes W_{\lambda_2\overline{\gamma_2}^{-1}\overline{\beta}}$ of degree bounded from above by 
    $$ \langle \varpi_i^{\vee}, \gamma_1+\beta\rangle + \deg \overline{\BT}_i^W(z)|_{W_{\lambda_2\overline{\gamma_2}^{-1}\overline{\beta}}} =\langle \varpi_i^{\vee}, \gamma_1+\beta\rangle + \langle \varpi_i^{\vee}, \gamma_2-\beta \rangle  = \langle \varpi_i^{\vee}, \gamma_1+\gamma_2\rangle. $$
\end{itemize}
It follows that $\overline{\BT}_i^{V\otimes W}(z) v$ is a finite sum of polynomials of degree bounded from above by $\langle \varpi_i^{\vee}, \gamma\rangle$. Its coefficient of $z^{\langle \varpi_i^{\vee},\gamma\rangle}$, denoted by $D_{i,\gamma}^{V\otimes W}(v)$, is $D_{i,\gamma_1}^V(v_1)\otimes D_{i,\gamma_2}^W(v_2)$ plus a finite sum over $\beta \in \BQ_+\setminus \{0\}$ of vectors in $V_{\lambda_1\overline{\gamma_1}^{-1}\overline{\beta}^{-1}} \otimes W_{\lambda_2\overline{\gamma_2}^{-1}\overline{\beta}}$. From invertibility of all the $D_{i,\gamma_1}^V$ and $D_{i,\gamma_2}^W$ we obtain the invertibility of $D_{i,\gamma}^{V\otimes W}$.
\end{proof}
The proof of invertibility in part (a) used the $\BT_i^-(z)$ and hence Drinfeld--Cartan generators $h_{j,s}$ for $s<0$. It does not work for highest $\ell$-weight modules of the Borel subalgebra. For example, when $\Glie = \mathfrak{sl}_2$, the infinite-dimensional simple $\Borel$-module $L(1-z)$ is not T-polynomial because $\BT_1^+(z)$ acts identically as 1; see \cite[\S 7.1]{HJ}.

\subsection{Global A-polynomiality}

We give a sufficient condition for a module in category $\Osh$ to descend to a truncation. For that purpose, we need to pass from local A-polynomiality to global A-polynomiality.

\begin{prop}  \label{prop: key}
    Let $(\mu, \Bp, \mathbf{z})$ be a triple of coweight, polynomial $\ell$-weight and constant $\ell$-weight satisfying Eq.\eqref{truncation: new parameter}.
    Let $V$ be a T-polynomial module in category $\mathcal{O}_{\mu}$. Assume that for each $i \in I$, the image of the generating series $\SA_i^+(z) := \Bp_i^*(z) \BA_i^+(z)$ in $(\mathrm{End}V)[[z]]$ stabilizes $V[z]$. Then 
    \begin{itemize}
        \item[(i)] The pair $(\mu, \Bp)$ is truncatable in the sense of Definition \ref{defi: truncation} and the $\CU_{\mu}(\Gaff)$-module structure on $V$ factorizes through the quotient map $\pi_1: \CU_{\mu}(\Gaff) \rightarrow \CU_{\mu}^{\Bp}(\Gaff)$.
        \item[(ii)] There exist $\lambda' \in \mathfrak{p}_0$ such that the $\CU_{\mu}(\Gaff)$-module $L_0(\lambda') \otimes V$ is the pullback along the algebra homomorphism $\pi_3 \circ \iota: \CU_{\mu}(\Gaff) \rightarrow \CU_{\mu}^{\Bp,\mathbf{z}}(\Gaff)$ of a certain $\, \CU_{\mu}^{\Bp,\mathbf{z}}(\Gaff)$-module.
        \item[(iii)] The $\CU_{\mu}(\Gaff)$-module $V$ is of finite representation length.
    \end{itemize}
\end{prop}
\begin{proof}
    (i) By assumption, the eigenvalue of the generating series $\SA_i^+(z)$ acting on the one-dimensional top weight space $V_{\lambda}$ is a polynomial $p_i(z)$ of constant term 1. Write the leading term of $p_i(z)$ as $u_i z^{r_i}$ with $u_i \in \BC^{\times}$ and $r_i \in \BN$. Combining Eq.\eqref{rel: A and T} with Definition \ref{defi: f} we have 
    $$ \SA_i^+(z) \overline{\BT}_i^V(z) = 
 p_i(z) \overline{\BT}_i^V(zq_i^{-2}) \in (\mathrm{End}V)[[z]].  $$
 Restricting the above equation to a weight space $V_{\lambda
 \overline{\beta}^{-1}}$ with $\beta \in \BQ_+$, we get an equality of polynomials with coefficients in the finite-dimensional algebra $\mathrm{End}(V_{\lambda\overline{\beta}^{-1}})$. Let $z^n E_{i,\beta}^V$ and $z^{\langle \varpi_i^{\vee}, \beta\rangle} D_{i,\beta}^V$ denote the leading terms of $\SA_i^+(z)$ and $\overline{\BT}_i^V(z)$ acting on $V_{\lambda\overline{\beta}^{-1}}$. By comparing the leading terms at both sides, we get
 $$E_{i,\beta}^V z^n D_{i,\beta}^Vz^{\langle\varpi_i^{\vee},\beta\rangle}  =  u_i z^{r_i}  D_{i,\beta}^V (zq_i^{-2})^{\langle\varpi_i^{\vee},\beta\rangle} \in \mathrm{End} (V_{\lambda\overline{\beta}^{-1}})[z]. $$
Since $D_{i,\beta}^V$ is invertible, the polynomial $\SA_i^+(z)|_{V_{\lambda\overline{\beta}^{-1}}}$ is of degree $n = r_i$ independent of $\beta$, and its leading coefficient $E_{i,\beta}^V$ is the nonzero scalar $u_i q_i^{-2\langle\varpi_i^{\vee},\beta\rangle }$. 

As a consequence, the generating series $\SA_i^+(z)$ acting on $V$ is a polynomial of fixed degree $r_i$ with invertible leading coefficient. We apply Corollary \ref{cor: truncation criterion} to get (i). In particular, the degree $r_i$ of $\SA_i^+(z)$ is the coefficient $t_i$ in Eq.\eqref{truncatable} that comes from the truncatable pair $(\mu, \Bp)$. 

(ii) Write $\mu = \sum_{i\in I}m_i \varpi_i^{\vee}$.  We first show that the central elements $\phi_{i,0}^+\phi_{i,m_i}^-$ for $i \in I$ act on $V$ as scalars. Let $\lambda = (\lambda_i)_{i\in I}$ be the top weight of $V$. Let $\beta = \sum_{j\in I} s_j \alpha_j \in \BQ_+$ with $s_j \in \BN$. By the definition of weight, $\phi_{i,0}^+$ acts on the weight space $V_{\lambda \overline{\beta}^{-1}}$ as the scalar
$$ \lambda_i \prod_{j\in I} q^{-s_j b_{ji}} = \lambda_i \prod_{j\in I} q_j^{-s_j c_{ji}} =: \lambda_{i,\beta}. $$
From the proof of (i), we see that $\SA_{i,t_i}^+$ acts on the weight space as $E_{i,\beta}^V$, which is the scalar $u_i q_i^{-2s_i}$. 
Since $V$ is a module over $\CU_{\mu}^{\Bp}(\Gaff)$, from Eq.\eqref{truncation: simply connected}  we see that $\phi_{i,m_i}^-$ acts on the weight space as the scalar 
$$ \lambda_{i,\beta}  \Bp_{i,n_i} \Bp_{i,0}^{-1} \prod_{j\in I} (u_j q_j^{-2s_j} q_j^{t_j})^{-c_{ji}} = \lambda_i \Bp_{i,n_i} \Bp_{i,0}^{-1} \prod_{j\in I} (u_j  q_j^{t_j})^{-c_{ji}} \times \prod_{j\in I} q_j^{s_j c_{ji}}. $$
It follows that $\phi_{i,0}^+ \phi_{i,m_i}^-$ acts on the weight space $V_{\lambda \overline{\beta}^{-1}}$ as the following scalar 
$$ \lambda_i^2 \Bp_{i,n_i} \Bp_{i,0}^{-1} \prod_{j\in I} (u_j  q_j^{t_j})^{-c_{ji}} =: v_i. $$
Clearly $v_i$ is independent of $\beta$. So $\phi_{i,0}^+ \phi_{i,m_i}^-$ acts on $V$ as the scalar $v_i$.

Next, for each $i \in I$ choose a square root $\lambda_i'$ of
$$ v_i^{-1} \Bp_{i,0}^{-1} \Bp_{i,n_i} \prod_{j\in I} (-q_j)^{-t_jc_{ji}} = \lambda_i^{-2} \prod_{j\in I} ((-1)^{t_j} u_j)^{c_{ji}} .  $$
Take $\lambda' = (\lambda_i')_{i\in I} \in \mathfrak{t}^*$ and $V' := L_0(\lambda') \otimes V$. One verifies directly that the $\CU_{\mu}^{\Bp}(\Gaff)$-module structure on $V'$ factorizes through the quotient map $\pi_{12}: \CU_{\mu}^{\Bp}(\Gaff)\rightarrow \overline{\CU}_{\mu}^{\Bp}(\Gaff)$. Since tensor product with an invertible module does not change the action of the series $\SA_i^{\pm}(z)$, the computation of the action of $\SA_{i,t_i}^+$ of (i) still makes sense in $V'$, which implies that $V'$ is balanced in the sense of Definition \ref{defi: balance}. We apply Proposition \ref{prop: balace} to get (ii).

 (iii) Combining (ii) with Theorem \ref{thm: JH}, we see that the $\CU_{\mu}(\Gaff)$-module $L_0(\lambda') \otimes V$ is of finite representation length in category $\mathcal{O}_{\mu}$. The same holds true for $V$ as the tensor product functor $L_0(\lambda') \otimes -$ an auto-equivalence of category $\mathcal{O}_{\mu}$.
\end{proof}

\subsection{Main result} Recall the monoid automorphism $\Bp \mapsto \widetilde{\Bp}$ of $\mathfrak{p}$ defined before Theorem \ref{thm: polynomiality A series}.

\begin{theorem}  \label{thm: truncation shifted}
Let $s\geq 1$ and $\Bm_k, \Bn_k \in \mathfrak{p}$ be polynomial $\ell$-weights of coweights $\mu_k$ and $\nu_k$ respectively for $1\leq k \leq s$. Define the coweight $\mu$ and the polynomial $\ell$-weight $\Bp$ by
$$ \mu = \mu_1+\mu_2+\cdots+\mu_s - \nu_1-\nu_2-\cdots-\nu_s,\quad \Bp :=  \Bm_1 \Bm_2 \cdots \Bm_s \widetilde{\Bn_1} \widetilde{\Bn_2} \cdots \widetilde{\Bn_s}. $$
 Then the pair $(\mu, \Bp)$ is truncatable. Moreover, for any constant $\ell$-weight $\mathbf{z}$ satisfying Eq.\eqref{truncation: new parameter}, the following statements hold true.
\begin{itemize}
    \item[(i)] The following fusion product module over $\CU_{\mu}(\Gaff)$ is of finite length in category $\mathcal{O}_{\mu}$ and, up to tensor product by an invertible module, it is the pullback of a certain $\,\CU_{\mu}^{\Bp,\mathbf{z}}(\Gaff)$-module along the algebra homomorphism $\pi_3\circ \iota: \CU_{\mu}(\Gaff) \rightarrow \CU_{\mu}^{\Bp,\mathbf{z}}(\Gaff)$:
    $$ L_{\mu_1-\nu_1}\left(\frac{\Bm_1}{\Bn_1}\right) \ast L_{\mu_2-\nu_2}\left(\frac{\Bm_2}{\Bn_2}\right) \ast \cdots \ast L_{\mu_s-\nu_s}\left(\frac{\Bm_s}{\Bn_s}\right).  $$
    \item[(ii)] Assume that $\frac{\Bm_k}{\Bn_k} \in \widehat{\mathfrak{r}}$ for $1\leq k \leq s$. The following tensor product module over $\qaf$ is of finite length in category $\widehat{\mathcal{O}}$ and, up to tensor product by an invertible module, it is the pullback of a certain $\,\CU_0^{\Bp,\mathbf{z}}(\Gaff)$-module along the algebra homomorphism $\pi_3\circ \iota: \CU_0(\Gaff) \rightarrow \CU_0^{\Bp,\mathbf{z}}(\Gaff)$:
    $$ L_0\left(\frac{\Bm_1}{\Bn_1}\right) \otimes L_0\left(\frac{\Bm_2}{\Bn_2}\right) \otimes \cdots \otimes L_0\left(\frac{\Bm_s}{\Bn_s}\right). $$
\end{itemize}
\end{theorem}

\begin{proof}
By Proposition \ref{prop: poly T}, both the fusion product (as a highest $\ell$-weight module in category $\Osh$) and the tensor product are T-polynomial. By Proposition \ref{prop: key}, it suffices to show that for such a module $V$, the generating series $\Bp_i^*(z) \BA_i^+(z)$ send $V$ to $V[z]$. 

(i) Since the action of $\BA_i^+(z)$ is diagonal under a fusion product and $\Bp^*$ is multiplicative with respect to $\Bp$, one may assume $s=1$ consider a single irreducible module $V = L_{\mu}(\frac{\Bm_1}{\Bn_1})$. Notice that $V$ is an irreducible subquotient of the tensor product of the one-dimensional module $L_{\mu_1}(\Bm_1)$ with another module $L_{-\nu_1}(\frac{1}{\Bn_1})$. This tensor product is compatible with truncation by Eq.\eqref{rel: tensor product trivial}, so we may assume further $\Bm_1 = 1$.

    Since the coweight $\mu$ is antidominant, we can regard the $\CU_{\mu}(\Gaff)$-module $V$ as a $\Borel$-module by the algebra homomorphism $\jmath_{\mu}$ of Proposition \ref{prop: embedding}. By Theorem \ref{thm: Borel O} this gives us $L(\frac{1}{\Bn_1})$. Apply Theorem \ref{thm: polynomiality A series} to this $\Borel$-module and notice that $\jmath_{\mu}$ preserves the generating series $\BA_i^+(z)$, we see that each $\Bp_i^*(z) \BA_i^+(z) \in \CU_{\mu}(\Gaff)[[z]]$ maps $V$ to $V[z]$. 

 (ii) For the tensor product module, one uses the Hopf algebra inclusion $\Borel \subset \qaf$ to view $V$ as a tensor product of $\Borel$-modules. Then Theorem \ref{thm: polynomiality A series} applies.
\end{proof}
It follows from Theorem \ref{thm: truncation shifted} and Proposition \ref{prop: key} that in category $\mathcal{O}_{\mu}$: each simple module factorizes through a certain $\CU_{\mu}^{\Bp}(\Gaff)$; up to tensor product with an invertible module each simple module factorizes through a certain $\overline{\CU}_{\mu}^{\Bp}(\Gaff)$.

We expect the Drinfeld--Jimbo coproduct of the quantum affine algebra to be compatible with truncations.
Given a polynomial $\ell$-weight $\Bp$ with $(0, \Bp)$ truncatable, we define $U_q^{\Bp}(\Gaff)$ to be the quotient algebra of $\qaf$ by the relations \eqref{truncation: poly}--\eqref{truncation: simply connected} and let $\pi^{\Bp}: \qaf \rightarrow U_q^{\Bp}(\Gaff)$ be the quotient map. For $\mathbf{b}$ another such $\ell$-weight, we expect the existence of an algebra homomorphism $\Delta^{\Bp,\mathbf{b}}: U_q^{\Bp\mathbf{b}}(\Gaff) \rightarrow U_q^{\Bp}(\Gaff) \otimes U_q^{\mathbf{b}}(\Gaff)$ that makes the following diagram commutative:
  \begin{gather*}
 \xymatrixcolsep{6pc} \xymatrix{
\qaf \ar[d]_{\pi^{\Bp\mathbf{b}}} \ar[r]^{\Delta} & \qaf \otimes \qaf \ar[d]^{\pi^{\Bp}\otimes \pi^{\mathbf{b}}}  \\
U_q^{\Bp\mathbf{b}}(\Gaff) \ar@{-->}[r]^{\Delta^{\Bp,\mathbf{b}}}          & U_q^{\Bp}(\Gaff) \otimes U_q^{\mathbf{b}}(\Gaff) }  
\end{gather*}
The existence of $\Delta^{\Bp,\mathbf{b}}$ is equivalent to a polynomiality property of $\Delta(\BA_i^{\pm}(z))$. Theorem \ref{thm: truncation shifted}(ii) should be the above compatibility evaluated at properly chosen tensor product representations. Our proof uses the universal R-matrix to compute the action of $\Delta(\BA_i^+(z))$ on these representations and is much in the spirit of the R-matrix realization of the quantum affine algebra \cite{FRT}. 

\begin{example}
    Let $\Glie = \mathfrak{sl}_2$. Fix $a, b \in \BC^{\times}$ and consider the simple $\CU_{-\varpi_1^{\vee}}(\Gaff)$-module $L_{-\varpi_1^{\vee}}(\frac{b}{1-za})$. It is isomorphic to the tensor product $L_0(b) \otimes L_{-\varpi_1^{\vee}}(\frac{1}{1-za})$. It has a basis $(w_n)_{n\in \BN}$ with respect to which (see \cite[Example 11.19]{H})
    $$ \phi_1^{\pm}(z) w_n = b q^{-2n} \frac{1-zaq^2}{(1-zaq^{-2n})(1-zaq^{2-2n})} w_n.  $$
    One verifies directly that the module structure on $L_{-\varpi_1^{\vee}}(\frac{b}{1-za})$ factorizes through the quotient $\pi_1: \CU_{-\varpi_1^{\vee}}(\Gaff) \rightarrow \CU_{-\varpi_1^{\vee}}^{1-zaq^2}(\Gaff)$ with the action of A-series given by
    $$ \SA_1^+(z) w_n = (1-zaq^{-2n}) w_n,\quad \SA_1^-(z) w_n = (1 - z^{-1} a^{-1} q^{2n}) w_n. $$ 
    The module structure factorizes through the quotient $\pi_{12}\circ \pi_1: \CU_{-\varpi_1^{\vee}}(\Gaff) \rightarrow  \overline{\CU}_{-\varpi_1^{\vee}}^{1-zaq^2}(\Gaff)$ if and only if $b^2 = a^2$. Let $\mathbf{z}$ be a square root of $-aq^2$ so that the triple $(-\varpi_1^{\vee}, 1-zaq^2, \mathbf{z})$ satisfies Eq.\eqref{truncation: new parameter}. Then, the module is the pullback of a $\CU_{-\varpi_1^{\vee}}^{1-zaq^2,\mathbf{z}}(\Gaff)$-module along the algebra homomorphism $\pi_3\circ \iota: \CU_{-\varpi_1^{\vee}}(\Gaff) \rightarrow \CU_{-\varpi_1^{\vee}}^{1-zaq^2,\mathbf{z}}(\Gaff)$ if and only if $b = a$. 
\end{example} 
\section{Applications}\label{appli}

 We discuss several applications of our main result: subring structures for subcategories of finite length representations in the topological Grothendieck ring (Corollaries \ref{subring}, \ref{fls} and \ref{subring cluster}),  their relation to the cluster algebras of \cite{GHL} (Conjecture \ref{isomclus}) and descent to truncation in accordance with \cite[Conjecture 12.2]{H}(Corollary \ref{clastrunc}).

\subsection{Grothendieck rings}

Let $\mathcal{O}_\mu^{\rm f}$ be the category of modules in 
$\mathcal{O}_\mu$ of finite length. Its Grothendieck group $K_0(\mathcal{O}_{\mu}^{\rm f})$ is the free abelian group generated by the isomorphism classes $[L_{\mu}(\Bf)]$ for $\Bf \in \mathfrak{r}_{\mu}$. It admits a natural completion $K_0(\mathcal{O}_\mu)$ which we will view as the the Grothendieck group of the category $\mathcal{O}_\mu$ (this is the version used in \cite{H}).
Let $K_0(\mathcal{O}^{\rm sh,f})$ be the direct sum, over all coweights, of $K_0(\mathcal{O}_\mu^{\rm f})$. Then the completion 
$$K_0(\Osh) = \bigoplus_{\mu \in \BP^{\vee}} K_0(\mathcal{O}_{\mu})$$ 
of $K_0(\mathcal{O}^{\rm sh,f})$ is endowed with a commutative ring structure by fusion product \cite[\S 5.4]{H}:
$$ [L_{\mu}(\Be)] [L_{\nu}(\Bf)] = [L_{\mu}(\Be) \ast L_{\nu}(\Bf)] \quad \textrm{for $\Be \in \mathfrak{r}_{\mu}$ and $\Bf \in \mathfrak{r}_{\nu}$}. $$
We shall refer to $K_0(\Osh)$ as topological Grothendieck ring.
As a consequence of our main Theorem \ref{thm: truncation shifted}, we obtain the following.
\begin{cor}\label{subring} The subgroup $K_0(\mathcal{O}^{\rm sh,f})$ of the topological Grothendieck ring $K_0(\Osh)$ is a subring.
\end{cor}

As above, the Grothendieck group $K_0(\widehat{\mathcal{O}}^{\rm f})$ of the category $\widehat{\mathcal{O}}^{\rm f}$ of finite length $\qaf$-modules in $\widehat{\mathcal{O}}$ can be completed to a group $K_0(\widehat{\mathcal{O}})$. 
This completed group has a commutative ring structure induced by the Drinfeld--Jimbo coproduct, or equivalently by the fusion product as a consequence of the injective $q$-character map \cite{FR1}. The resulting topological Grothendieck ring $K_0(\widehat{\mathcal{O}})$ is a subring of $K_0(\Osh)$ for the same reason.

\begin{cor}\label{fls} The subgroup $K_0(\widehat{\mathcal{O}}^{\rm f})$ of the topological Grothendieck ring $K_0(\widehat{\mathcal{O}})$ is a subring.
\end{cor}

\begin{rem} The picture is very different in the case of the Borel subalgebra $\Borel$ of the quantum affine algebra.
Consider the Grothendieck group $K_0(\mathcal{O}^{\rm f})$ and the topological Grothendieck ring $K_0(\mathcal{O})$ as above, but for representations of the Borel algebra. It is known from \cite[Remark 5.12]{HL} that the subgroup $K_0(\mathcal{O}^{\rm f})$ is not a subring. For example, in the $\mathfrak{sl}_2$-case, the tensor product of a positive prefundamental representation (associated with a prefundamental $\ell$-weight) by a negative prefundamental representations (associated with the inverse of a prefundamental $\ell$-weight) has infinite length. 
Note however that certain important subcategories of finite length representations of the Borel algebra are stable under tensor product; see \cite[Proposition 5.11]{HL}.
\end{rem}

\subsection{Cluster algebras} The relation between representations of shifted quantum affine algebras and cluster algebras was investigated in \cite{GHL}.

Fix a square root $\sqrt{q}$ of $q$ and define $\varpi_i'$ to be the constant $\ell$-weight $(\sqrt{q}^{\delta_{ij}d_i} )_{j\in I}$. Then $(\varpi_i')^2$ is precisely $\varpi_i$ defined in Subsection \ref{subsection: O qaf}. 
In \cite{GHL} the following remarkable subcategory $\mathcal{O}_{\mathbb{Z}}^{\rm sh}$ of $\Osh$ is introduced. It is the abelian subcategory of modules in $\Osh$ whose simple constituents have as highest $\ell$-weights Laurent monomials in variables $\Psi_{i,q^r}$ and $\varpi_j'$, for $i, j \in I$ and $r\in\mathbb{Z}$, with an additional condition on the couple $(i,r)$; see \cite[Definition 9.13]{GHL}. The category $\mathcal{O}^{\rm sh}_{\mathbb{Z}}$ is stable by fusion product so that its topological Grothendieck ring $K_0(\mathcal{O}_{\BZ}^{\rm sh})$ is well-defined as a subring of $K_0(\Osh)$. 

Let $\mathcal{O}_{\BZ}^{\rm sh,f}$ be the full subcategory of $\Osh_{\BZ}$ consisting of finite length representations. Consider its Grothendieck group $K_0(\mathcal{O}_{\BZ}^{\rm sh,f})$ as a subgroup of $K_0(\mathcal{O}_{\BZ}^{\rm sh})$.
The following result is a direct consequence of Corollary \ref{subring}.

\begin{cor}\label{subring cluster} The subgroup $K_0(\mathcal{O}^{\rm sh,f}_{\mathbb{Z}})$ of the topological Grothendieck ring $K_0(\Osh_{\mathbb{Z}})$ is a subring.
\end{cor}

It is proved in \cite[Theorem 9.15]{GHL} that there exists an injective ring morphism
$$F : \mathcal{A} \rightarrow K_0(\mathcal{O}_\mathbb{Z}^{\rm sh}).$$
Here $\mathcal{A}$ is a cluster algebra of infinite rank introduced in \cite{GHL}; to be precise our $\mathcal{A}$ and $K_0(\Osh_{\BZ})$ correspond to $[\underline{P}] \otimes_{[P]} \SA_{w_0}$ and $[\underline{P}] \otimes_{[P]} K_0(\mathscr{O}_{\BZ})$ in {\it loc.cit.} Moreover, it is prove that the topological completion of the image of $F$ is the whole topological Grothendieck ring $K_0(\mathcal{O}_\mathbb{Z}^{\rm sh})$. 

Recall that a cluster algebra is generated as ring by cluster variables. It is conjectured in \cite{GHL} that all cluster variables of $\mathcal{A}$ are mapped by $F$ to certain isomorphism classes of simple objects in $\Osh_{\mathbb{Z}}$. If this is true, using Corollary \ref{subring cluster}, we get that the image of $F$ is included in the subring $K_0(\mathcal{O}^{\rm sh, f}_{\BZ})$.


We are now in a position to formulate the following conjecture, which gives a finer categorical interpretation of the cluster algebra $\mathcal{A}$.

\begin{conj}\label{isomclus} The ring morphism $F$ induces a ring isomorphism
$$\mathcal{A} \simeq K_0(\mathcal{O}^{\rm sh, f}_{\mathbb{Z}}).$$
\end{conj}

This is established for $\mathfrak{g} = \mathfrak{sl}_2$ in \cite[Remark 9.32]{GHL}. This is an isomorphism of rings (not of topological rings).

\subsection{Classification of simple modules by Langlands dual $q$-character}

In \cite{H} the first author proposed a general conjecture to classify simple representations of truncated shifted quantum affine algebras in category $\Osh$ in terms of Langlands dual $q$-characters. To each polynomial $\ell$-weight ${\bf a}$ is associated a representation $V^L$ of the Langlands dual quantum affine algebra $U_q(\hat{\mathfrak{g}}^L)$ which has a \lq\lq Langlands dual $q$-characters" $\chi_q^L(V^L)$ defined in \cite{H} (and that can be computed algorithmically). Then, the monomials of $\chi_q^L(V^L)$ are conjectured to give the highest $\ell$-weights of simple representations up to sign twist of the corresponding adjoint truncation (note that the statement of this conjecture is adjusted in \cite[Section 11]{H2}). 

Since the adjoint truncation and the intermediate truncation have the same classification (up to sign twist) of simple modules in category $\Osh$, the conjecture of \cite{H} can be equivalently stated for the intermediate truncation. Indeed, it suffices to replace in \cite[Conjecture 12.2]{H} the category $\mathcal{O}_{\mu,\mathcal{Z}}^{\lambda}$ by the subcategory $\overline{\mathcal{O}}_{\mu}^{\mathbf{Z}}$ of modules in $\mathcal{O}_{\mu}$ that factor through the intermediate truncation $\overline{\CU}_{\mu}^{\mathbf{Z}}(\Gaff)$. Then we can omit the parameters $z_i'$ in the collection $\mathcal{Z} = (\mathbf{Z}, z_1',z_2',\cdots,z_n')$.

The statement of Theorem \ref{thm: truncation shifted} with one factor $s = 1$ gives the following.

\begin{cor}\label{clastrunc} Any simple module in category $\mathcal{O}_\mu$  descends to a simply connected truncation, with the truncation parameters as 
predicted by \cite[Conjecture 12.2]{H}.\end{cor}

This generalizes the result obtained in \cite{H} for simple finite-dimensional representations in $\mathcal{O}_\mu$ by a different method.

\begin{example}\label{exneg} Recall the prefundamental $\ell$-weight in Example \ref{defi: prefund}.
Consider a negative prefundamental representation 
$L_{-\varpi_i^\vee}(\Psi_{i,a}^{-1})$ in $\mathcal{O}_{-\varpi_i^\vee}$. 
According to Theorem \ref{thm: truncation shifted}, up to tensor product with an invertible module this representation descends to the intermediate truncation $\overline{\mathcal{U}}_{-\varpi_i^\vee}^{\Bp}(\hat{\mathfrak{g}})$ where 
${\bf a} = \Psi_{\overline{i},aq^{r^\vee h^{\vee}}}$. In the notations of \cite{H}, this means this representation 
belongs to the category $\mathcal{O}_{-\varpi_i^\vee}^{\varpi_{\overline{i}}^\vee,\Bp}$. But $Z_{i,a^{-1}}^{-1}$ is the lowest weight term in the Langlands dual $q$-character 
$\chi_q^L(V_{\overline{i}}(a^{-1}q^{-r^\vee h^{\vee}}))$ 
of the fundamental representation $V_{\overline{i}}(a^{-1}q^{-r^\vee h^{\vee}})$ of the twisted quantum affine
algebra $U_q(\hat{\mathfrak{g}}^L)$; see \cite[Remark 12.3 (v)]{H0}\footnote{In the construction of \cite{FHR}, this is true up to the positivity Conjecture 6.11 therein.}. The two statements coincide as predicted by \cite[Conjecture 12.2]{H}.
\end{example}

\begin{proof} In the proof of 
\cite[Theorem 12.8]{H} for finite-dimensional representations, a reduction to (finite-dimensional) fundamental representations is used. Similarly, it suffices to prove the result for prefundamental representations as any simple module in a category $\mathcal{O}_\mu$ is a subquotient of a fusion product of prefundamental modules \cite[Corollary 5.6]{H} (up to twist by invertible modules). Then the truncations are compatible with fusion products \cite[Proposition 12.5]{H}. The result is clear for positive prefundamental representations as from \cite[Example 11.3]{H}, as $L_{\varpi_i^\vee}(\Psi_{i,a})$ belongs to $\mathcal{O}_{\varpi_i^\vee}^{\varpi_i^\vee,\Psi_{i,a}}$. The result for negative prefundamental representations follows from our main Theorem as explained in Remark \ref{exneg}.
\end{proof}

\begin{example} Les $\mu = \varpi_i^\vee - \varpi_j^\vee$ and $\Psi = \Psi_{i,a}\Psi_{j,b}^{-1}$ for some $i,j\in I$ and $a,b\in\mathbb{C}^\times$. Note that, in general, $L_\mu(\Psi)$ is not isomorphic to the fusion product of $L_{i,a}^+$ by $L_{j,b}^-$, but is isomorphic to a quotient of the fusion product. Now, $L_\mu(\Psi)$ is in the category $\mathcal{O}_{\mu}^{\varpi_i^\vee + \varpi_{\overline{j}}^\vee,\Bp}$ where $\Bp = \Psi_{i,a}\Psi_{\overline{j},bq^{r^\vee h^{\vee}}}$. This corresponds to the fact that $Z_{i,a^{-1}}Z_{j,b^{-1}}^{-1}$ is a monomial occurring in the Langlands dual $q$-character of the module 
$V_i(a^{-1})\otimes V_{\overline{j}}(b^{-1}q^{-r^\vee h^{\vee}})$ of the Langlands dual quantum affine algebra $U_q(\hat{\mathfrak{g}}^L)$.
\end{example}

\medskip

{\it Conflict of Interest Statement :}
There is no potential conflict of interest concerning this paper. There are no data
associated with this manuscript.

\end{document}